\documentclass{amsart}
\usepackage{amssymb,amsfonts}
\usepackage[all,arc]{xy}
\usepackage{enumerate}
\usepackage{color,soul}
    \sethlcolor{white}
\usepackage{mathrsfs}
\usepackage{pinlabel}
\usepackage{mathptmx}
\usepackage{anyfontsize}
\usepackage{t1enc}
\usepackage{rotating}
\usepackage{xparse}
\usepackage{subcaption,graphicx}
\usepackage[section]{placeins}
\usepackage{xcolor}
\newtheorem{thm}{Theorem}[section]
\newtheorem{cor}[thm]{Corollary}
\newtheorem{prop}[thm]{Proposition}
\newtheorem{lem}[thm]{Lemma}

\usepackage{graphicx}
\theoremstyle{definition}
\newtheorem{defn}[thm]{Definition}

\theoremstyle{remark}
\newtheorem{rem}[thm]{Remark}
\newtheorem{obs}[thm]{Observation}

\newcommand{\ra}{\rightarrow}

\newcommand{\R}{\mathbb{R}}

\newcommand*\interior[1]{\mathring{#1}}
\makeatletter
\let\c@equation\c@thm
\makeatother
\numberwithin{equation}{section}

\title{On Keen Weakly Reducible Bridge Spheres}

\author{Puttipong Pongtanapaisan and Daniel Rodman}

\begin{document}

\begin{abstract}
A bridge sphere is said to be keen weakly reducible if it admits a unique pair of disjoint compressing disks on opposite sides. In particular, such a bridge sphere is weakly reducible, not perturbed, and not topologically minimal in the sense of David Bachman. In terms of Jennifer Schultens' width complex, a link in bridge position with respect to a keen weakly reducible bridge sphere is distance one away from a local minimum. In this paper, we give infinitely many examples of keen weakly reducible bridge spheres for links in \(b\) bridge position for $b \geq 4.$ 
\end{abstract}

\maketitle

\section{Introduction}
Suppose that we have a decomposition of the 3-sphere $S^3= V_+ \cup_{\Sigma} V_-$ where $V_+,V_{-}$ are 3-balls and $\Sigma$ is a 2-sphere. A link $L \subset S^3$ intersecting $\Sigma$ transversely is said to be in \textit{bridge position} with respect to $\Sigma$ if $L \cap V_+ = \alpha_+$ and $L \cap V_- = \alpha_-,$ where $\alpha_+,\alpha_-$ are $b$-strand trivial tangles. The punctured sphere $\Sigma_L = \Sigma \backslash L$ is called a $b$-\textit{bridge sphere}. To each bridge sphere, we can assign a disk complex $\mathcal{D}(\Sigma_L)$, which is a simplicial complex whose vertices are isotopy classes of compressing disks in \(S^3\backslash L\) for $\Sigma_L$ and whose $k$ simplices are spanned by $k+1$ vertices with pairwise disjoint representatives.

We say that $\Sigma_L$ is \textit{topologically minimal} if one of the following holds:
\begin{enumerate}
    \item $\mathcal{D}(\Sigma_L) = \emptyset$.
    \item There exists $i \in \mathbb{N} \cup \lbrace 0 \rbrace$ such that the $i$th homotopy group of $\mathcal{D}(\Sigma_L)$ is nontrivial.
\end{enumerate}
The \textit{topological index} of $\Sigma_L$ is defined to be 0 if $\mathcal{D}(\Sigma_L) = \emptyset$, or the smallest $i$ such
that $\pi_{i-1}(\mathcal{D}(\Sigma_L))$ is nontrivial if $\mathcal{D}(\Sigma_L) \neq \emptyset$. The notion of topological minimality was introduced by David Bachman \cite{bachman2010topological} as a generalization of useful concepts such as incompressibilty and strong irreducibility of surfaces in a 3-manifold. It turns out that topologically minimal surfaces possess desirable properties. For instance, in an irreducible 3-manifold, a topologically minimal surface can be isotoped to intersect an incompressible surface in such a way that any intersection loop is essential in both surfaces. Furthermore, the concept of topological minimality gave rise to 
examples of 3-manifolds containing arbitrarily many non-minimal genus, unstabilized Heegaard surfaces that are weakly reducible \cite{bachman2013stabilizing}. In \cite{moriah2007heegaard}, Moriah dubbed these examples "the nemesis of Heegaard splittings" as they are difficult to find.

Conjecturally, there is a special and mysterious relationship between topologically minimal surfaces and geometrically minimal surfaces, which are surfaces whose mean curvature is identically zero. 
Every geometrically minimal surface has a Morse index, which roughly speaking counts the maximal number of directions the surface can be deformed so as to decrease its area.
Freedman, Hass and Scott showed that every surface of topological index zero is isotopic to a geometrically minimal surface of Morse index zero \cite{freedman1983least}. 
By works of Pitts and Rubinstein \cite{pitts1987applications} and of Ketover, Liokumovich, and Song \cite{ketover2019existence}, a Heegaard surface of topological index one is isotopic to a geometrically minimal surface of Morse index at most one. In \cite{campisi2019disk}, Campisi and Torres showed that the genus two Heegaard surface of the 3-sphere has topological index three. 
By Urbano's work, this Heegaard surface must have Morse index at least six \cite{urbano1990minimal}. 
Thus, it is not true in general that a surface of topological index $k$ is isotopic to a surface of Morse index at most $k,$ but the precise connection is not well-understood.

One can ask the interesting question of which surfaces are topologically minimal. Several authors have given examples of topologically minimal Heegaard surfaces \cite{bachman2010existence,campisi2018hyperbolic,campisi2019disk,lee2015topologically} and bridge surfaces \cite{lee2016bridge,pongtanapaisan2020critical,rodman2018infinite}. Heegaard surfaces that are not topologically minimal have also been studied by several authors who constructed \textit{keen weakly reducible} Heegaard surfaces.
That is, each of these surfaces possesses a unique \textit{weak reducing pair}, a pair of compressing disks on opposite sides of the surface whose boundaries are disjoint. 
By a result of McCullough \cite{mccullough1991virtually}, the disk complex of the boundary of a handlebody is contractible.
Thus having a unique pair of weak reducing disks on distinct sides of a Heegaard splitting means that in the disk complex, there is a unique edge connecting the two contractible subcomplexes corresponding the two handlebodies, resulting in a contractible disk complex for the Heegaard surface. The examples of keen weakly reducible Heegaard surfaces in the literature with simple descriptions include the canonical Heegaard surface of a surface bundle whose monodromy has sufficiently high translation distance by Johnson \cite{johnson2012mapping}, some Heegaard surfaces arising from self-amalgamations by E  \cite{qiang2014topologically}, and certain unstabilized genus three Heegaard surfaces in irreducible and orientable 3-manifolds by Kim \cite{kim2016topologically}. More complicated constructions of keen weakly reducible Heegaard surfaces of genus $g \geq 3$ can also be found in \cite{liang2018reducible,qiang2017keen}. 

The goal of this paper is to provide infinitely many examples of non-topologically minimal bridge spheres, which are lacking in the literature, by verifying that the canonical bridge sphere for certain links in plat position is keen weakly reducible. Such links are obtained by ``amalgamating'' two types of links whose canonical bridge spheres are topologically minimal. Keen weakly reducible bridge spheres also belong to a family of surfaces with finitely many pairs of disjoint compressing disks \cite{zhang2022heegaard}, which is interesting in its own right.

\begin{thm}\label{theorem:main}
There exist infinitely many links with keen weakly reducible bridge spheres.
\end{thm}

This paper is organized as follows. In Section 2, we discuss properties of a keen weakly reducible bridge sphere related to perturbations of bridge spheres, thin position of links, and essential surfaces in the link exterior. In Section 3, we define the notion of a plat position for a link, consider a particular family of links in plat position, and describe useful positions of curves on a punctured sphere with respect to a train track. In Section 4, we characterize the behaviors of curves that bound disks above or below the bridge sphere. In Section 5, we use a criterion presented in \cite{cho2008homeomorphisms} to show that keen weakly reducible bridge spheres are not topologically minimal.
\subsection*{Acknowledgements}
The authors would like to thank Roman Aranda, David Bachman, Ryan Blair, Charlie Frohman, and Maggy Tomova for helpful conversations. We thank the referees for finding a mistake in an earlier draft. Research conducted for this paper is supported by the Pacific Institute for the Mathematical Sciences (PIMS). The research and findings may not reflect those of the Institute.

\section{Consequences of being keen weakly reducible}
In this section, we discuss some consequences of putting a link in bridge position with respect to a keen weakly reducible bridge sphere. We remark that a priori keen weakly reducible bridge spheres are not necessarily canonical bridge spheres for links in plat position. 

\subsection{Unperturbed Bridge Spheres}
Let $L$ be a link in bridge position with respect to $\Sigma$. Then $L \cap V_+ = \alpha_+$ is a collection of disjoint embedded arcs with the property that there exists an isotopy (rel $\partial\alpha^+$) taking $\alpha_+$ into $\Sigma.$ For each arc $\alpha_+^i$ of $\alpha_+,$ the trace of such an isotopy is a disk called a \textit{bridge disk} $D^i_+$.
From each bridge disk \(D^i_+\) we can obtain a compressing disk \(dD^i_+\) called the \textit{frontier} of \(D^i_+\) using the construction \(dD_+^i=\left(\partial \overline{N}(D_+^i)\right)\cap V_+\). Analogous definitions can be made for $L \cap V_- = \alpha_-.$

We say that a bridge sphere $\Sigma_K$ is \textit{perturbed} if there exist two bridge disks $D_+^1 \subset V_+$ and $D_-^1 \subset V_-$ such that
$D_+^1 \cap D_-^1$ is a single point contained in $L$. It is an interesting problem to search for unperturbed bridge spheres for a link up to isotopy since a perturbed bridge can always be obtained from a bridge sphere that is not perturbed by an isotopy which introduces a maximal point and a minimal point as shown in Figure \ref{fig:perturbing}. In some cases, the only destabilized bridge sphere is the one that realizes the bridge number \cite{otal1985presentations,ozawa2011nonminimal,zupan2011properties}.
Another common way to show that a bridge sphere $\Sigma_L$ for a nontrivial link $L$ is unperturbed is to show that there is no weak reducing pair for $\Sigma_L$. 
Being keen weakly reducible implies the following.

\begin{figure}
    \centering
    \includegraphics[width=.7\textwidth]{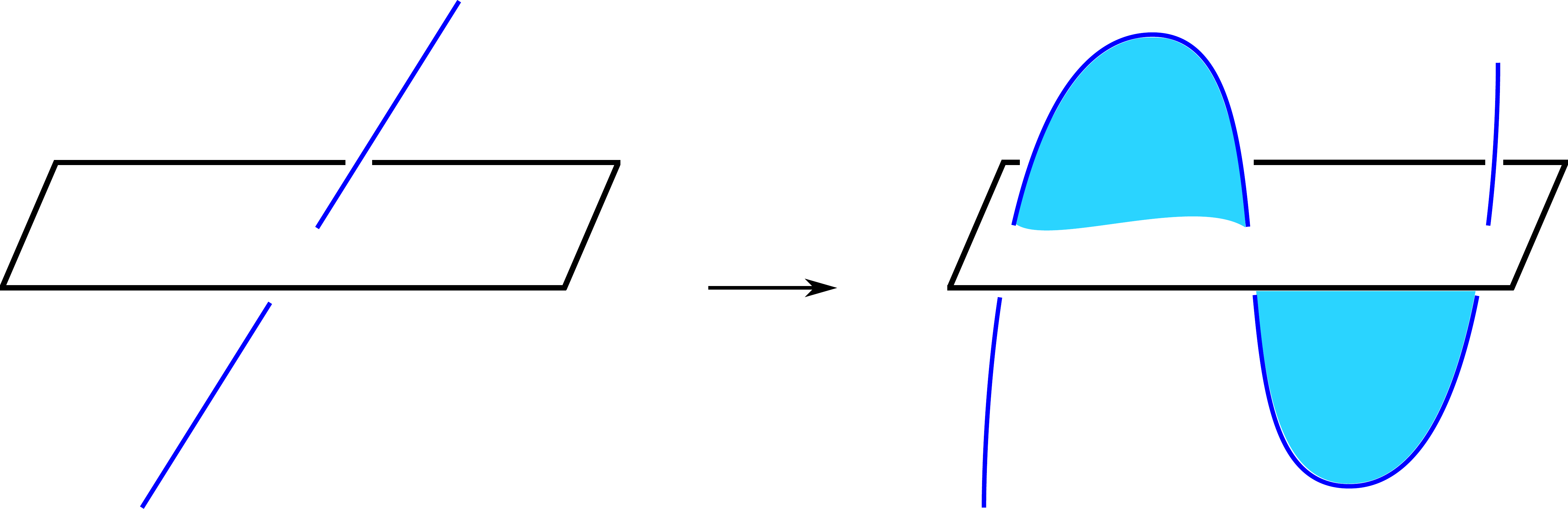}
    \caption{Introducing a cancelling pair of critical points.}
    \label{fig:perturbing}
\end{figure}

\begin{prop}\label{prop:nonperturbed}
If $\Sigma_L$ is keen weakly reducible, then $\Sigma_L$ is unperturbed.
\end{prop}

It is well known that a perturbed bridge sphere has a weak reducing pair; we prove that result here, and show that such a bridge sphere in fact has at least two weak reducing pairs.

\begin{proof}
If $\Sigma_L$ is a perturbed bridge sphere for a link $L$ in 2-bridge position, then $L$ must be the unknot and there exists a unique compressing disk $D$ above and a unique compressing disk $E$ below. Furthermore, $D \cap E \neq \emptyset,$ which implies that $\Sigma_L$ does not admit a weak reducing pair, and therefore $\Sigma_L$ cannot be keen weakly reducible. To complete the proof, we consider perturbed bridge spheres for links in $b$-bridge position, where $b\geq 3$.

\begin{figure}
    \centering
    \includegraphics[width=.7\textwidth]{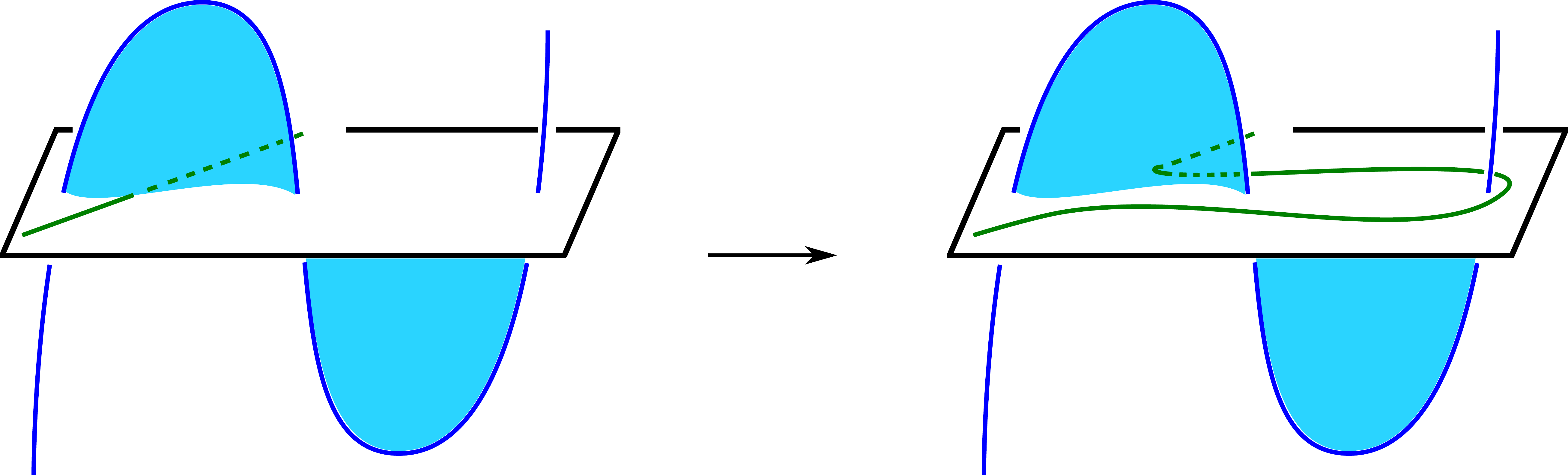}
    \caption{Bridge disks below $\Sigma_K$ can be isotoped to intersect $D_+^1$ in two points in $K$.}
    \label{fig:reroute}
\end{figure}

Suppose that $\Sigma_L$ is perturbed. By definition, there exist bridge disks $D_+^1 \subset V_+$ and $D_-^1 \subset V_-$ such that $D_+^1 \cap D_-^1 = \lbrace p \rbrace \in L.$
Let \(\mathcal{A}_+\) be a set of \(b\) disjoint bridge disks for \(\Sigma\), each corresponding to one of the components of \(\alpha_+\), and suppose further that \(D_+^1\in\mathcal{A}_+\).
Let \(\mathcal{A}_-\) be a similarly defined set of bridge disks below \(\Sigma\) with \(D_-^1\in\mathcal{A}_-\).
(We are able to define \(\mathcal{A}_+\) and \(\mathcal{A}_-\) \textit{after} $D_+^1$ and $D_-^1$ by Lemma 3.2 of \cite{scharlemann2005thin}.) The elements of \(\mathcal{A}_-\) may or may not intersect the interior of the arc \(D_+^1 \cap \Sigma\).
Below, we describe how if they do, we can replace them with another set of \(b\) disjoint bridge disks below \(\Sigma\), each of which is disjoint from the interior of \(D_+^1 \cap \Sigma\).

Suppose that the elements of  \(\mathcal{A}_-\) intersect the interior of $D_+^1 \cap \Sigma$. 
Consider a point \(q\) of intersection closest to $p$.
Let \(D_-'\in\mathcal{A}_-\) denote the bridge disk containing \(q\).
We perform a surgery on \(D_-'\) as depicted in Figure \ref{fig:reroute}, resulting in a new disk \(D_-''\).
Notice that \(D_-''\) is disjoint from the other elements of \(\mathcal{A}_-\), and \(D_-'\) and \(D_-''\) both correspond to the same bridge arc.
In slight abuse of notation, we will replace \(D_-'\) with \(D_-''\) in the collection \(\mathcal{A}_-\).
After this replacement, \(\mathcal{A}_-\) remains a collection of pairwise disjoint bridge disks for the bridge arcs below \(\Sigma\). 
The difference is that now, the elements of \(\mathcal{A}_-\) intersect the interior of \(D_+^1 \cap \Sigma\) in one fewer point.

We can repeatedly perform such surgeries until the bridge disks of \(\mathcal{A}_-\) are all disjoint from the interior of $D^1_+\cap\Sigma$.
It follows that the bridge disks of \(\mathcal{A}_-\) intersect \(D^1_+\) only in the two points of $D_+^1\cap L \cap \Sigma$, each of which intersects a bridge disk of \(\mathcal{A}_-\). Since \(L\) is in a \(b\) bridge position with \(b>2\), these two bridge disks must be distinct. In addition to these two, there must be at least one more bridge disk $D^2_-\in\mathcal{A}_-$ since \(b\geq 3\), and so \(D^2_-\) is disjoint from $D^1_+$.
Therefore, $dD_+^1$ and $dD^2_-$ comprise a weak reducing pair for $\Sigma_L.$ 

Now consider \(dD_-^1\).
We can mimic the trick in the previous paragraph so that a particular set of \(b\) pairwise disjoint bridge disks above $\Sigma$ intersects $D_-^1$ only in the two points of $D_-^1\cap L\cap \Sigma$. 
Then there is some bridge disk \(D^2_+\) disjoint from \(D_-^1\), which means that $dD_-^1$ and $dD^2_+$ comprise another weak reducing pair for $\Sigma_L.$ Therefore, a perturbed bridge sphere never admits a unique weak reducing pair and can never be keen weakly reducible.
\end{proof}

\subsection{Width complex}
Suppose that $L$ is a link and $h:S^3\rightarrow \mathbb{R}$ is the standard Morse function. Assume also that $\left.h\right|_L$ is a Morse function. Suppose that $c_1 < ... < c_n$ are critical values of $\left.h\right|_L$. Consider $h^{-1}(r_i)$, where $r_i$ is a regular value between $c_i$ and $c_{i+1}$. We say that a level sphere $h^{-1}(r_i)$ is a \textit{thin level} if $|h^{-1}(r_{i-1}) \cap L| > |h^{-1}(r_{i}) \cap L|$ and
$|h^{-1}(r_{i}) \cap L| < |h^{-1}(r_{i+1})  \cap L|.$ On the other hand, a level sphere $h^{-1}(r_i)$ is a \textit{thick level} if $|h^{-1}(r_{i-1}) \cap L| < |h^{-1}(r_{i}) \cap L|$ and
$|h^{-1}(r_{i}) \cap L| > |h^{-1}(r_{i+1})  \cap L|.$ We say that a disk $D\subseteq S^3\backslash L$ is a \textit{strong upper (resp.\ lower) disk} with respect to $h^{-1}(r_i)$ if 
\begin{enumerate}
    \item $\partial D = \alpha \cup \beta$ where $\alpha \subset L$ contains exactly one maximal (resp.\ minimal) point and $\beta$ is an arc in $h^{-1}(r_i)$, and
    \item the interior of $D$ does not intersect any level sphere $h^{-1}(r_j)$, where $1 \leq j \leq n-1.$
\end{enumerate}

If there exists a strong upper disk and a strong lower disk intersecting in exactly one point lying in $L$ (see Fig.\ \ref{fig:perturbing}, for instance), then there is an isotopy that cancels a maximal point and a minimal point. We call such a move a \textit{Type I move}. On the other hand, if there exists a strong upper disk and a strong lower disk that are disjoint, then there is an isotopy that interchanges a maximal point and a minimal point. We call such a move a \textit{Type II move}. 

In \cite{schultens2009width}, Schultens associated to a knot $K$ a graph called the \textit{width complex of $K$} to understand the structure of the collection of Morse embeddings of a fixed knot $K$. Two embeddings $k$ and $k'$ of $K$ are considered to be \textit{equivalent} if their thin and thick levels are isotopic. With this definition of equivalence, each vertex of the width complex is an equivalence class of embeddings of $K$ such that $\left.h\right|_K$ is a Morse function. An edge connects two vertices representing embeddings $k$ and $k'$ if $k$ differs from $k'$ by one of the following moves: a Type I move, the inverse of a Type I move, a Type II move, or the inverse of a Type II move. Schultens proved the following interesting result.

\begin{thm}[\cite{schultens2009width}]\label{thm:shultens2009width}
The width complex of a knot is connected.
\end{thm}

The proof of Theorem \ref{thm:shultens2009width} uses the fact that projections of $k$ and $k'$ to the vertical plane differ by a finite number of Reidemeister moves and planar isotopy. Furthermore, each of these local moves either affects an embedding by a Type I or a Type II move or does not alter the equivalence class at all. 
As any two projections of a multi-component link $L$ are also related by Reidemeister moves and planar isotopy, it follows that the width complex of a multi-component link is also connected.

A vertex that is particularly interesting is one representing an embedding that admits no Type I nor Type II moves. Such an embedding is said to be in \textit{locally thin position}.

\begin{prop}
Suppose that $l$ is an embedding of a link $L$ in bridge position with respect to a keen weakly reducible bridge sphere. In the width complex of $L$, there is an edge between $l$ and an embedding $l'$ of $L$ in a locally thin position.
\end{prop}
\begin{proof}
Let $D\subseteq V_+$ and $E\subseteq V_-$ be a weak reducing pair for a keen weakly reducible bridge sphere $\Sigma_L$.

\textbf{Claim:} $\partial D$ and $\partial E$ each cut out a twice punctured disk from $\Sigma_L$.
\begin{proof}[Proof of claim]
Suppose that $\partial D$ cuts $\Sigma_L$ into two components $F_1$ and $F_2$, where each component is a punctured disk containing more than two punctures. 
The loop $\partial E$ is contained in one of the components, say $F_1.$ There exists at least one bridge disk $D^1_+$ such that $\partial D^1_+ = \alpha \cup \beta$ where $\alpha \subset L$ and $\beta \subset F_2$. Then, $dD^1_+$ and $E$ give rise to a weak reducing pair distinct from $D$ and $E,$ which is a contradiction.
The same argument also implies \(\partial E\) cuts out a twice-punctured disk from \(\Sigma_L\).
\end{proof}
Observe that $D$ cuts off a 3-ball containing a unique bridge disk, which is a strong upper disk disjoint from a strong lower disk contained in a 3-ball cut off by $E.$ 
This pair of disks gives rise to a Type II move. After the Type II move is performed, there are neither Type I nor Type II moves left to perform because any pair of strong upper disk and strong lower disk (intersecting in one point of $L$ or mutually disjoint) that emerges after the Type II move on $D$ and $E$ will yield a distinct pair of strong upper disk and strong lower disk on $\Sigma_L$ and hence, $\Sigma_L$ admits more than one weak reducing pair, which is a contradiction.
\end{proof}

 After a Type II move is performed along $D$ and $E$, a thin level emerges. This thin level is incompressible because a compressing disk for this level would imply the existence of another weak reducing pair different from $D$ and $E.$ Thus, we obtain the following corollary.
\begin{cor}
A link with a keen weakly reducible bridge sphere contains an essential meridional surface in its exterior.
\end{cor}

\section{Setting}\label{sec:setting}

In this section, we redevelop and summarize several tools and concepts of Johnson and Moriah, all found in \cite{johnson2016bridge}.
Specifically, Subsection \ref{subsec:plat_projections} is a brief summary of Johnson and Moriah's plat links and accompanying tools such as their \(\sigma_i\) and \(\pi_y\) projection maps.
Then in Subsection \ref{subsec:Plat_train_tracks}, we develop Johnson and Moriah's taos, eyelets, and train tracks and adapt them slightly to our situation.
Finally in Subsection \ref{subsec:carried_almost_carried} we define the concepts of carried and almost carried arcs, loops, and graphs in a manner very similar to that of Johnson and Moriah, differing only in some minor ways that suit our purposes.

\subsection{Plat Positions}\label{subsec:plat_projections}

Consider the standard Morse function $h:S^3 \rightarrow \mathbb{R}$ with exactly one maximum, $+\infty$, and one minimum, $-\infty.$ Let $\alpha \subset S^3$ be a strictly increasing arc such that $\partial \alpha = \lbrace \pm \infty \rbrace.$ 
We identify $S^3 \backslash \alpha$ with $\mathbb{R}^3$ with Cartesian coordinates $(x,y,z)$ in such a way that the \(xz\)-plane lies in \(h^{-1}(0)\), and more generally, for each \(t\in\R\), the plane \(y=t\) lies in \(h^{-1}(t)\). We orient our perspective so that the $x$-axis is horizontal, the $y$-axis is vertical, and the $z$-axis points towards the reader. 
(This allows us to use terms like ``up,'' ``down,'' ``left,'' and ``right.'')
We denote $h^{-1}(t)$ as $P_t$.

\begin{figure}
    \centering
    \includegraphics[width=.6\textwidth]{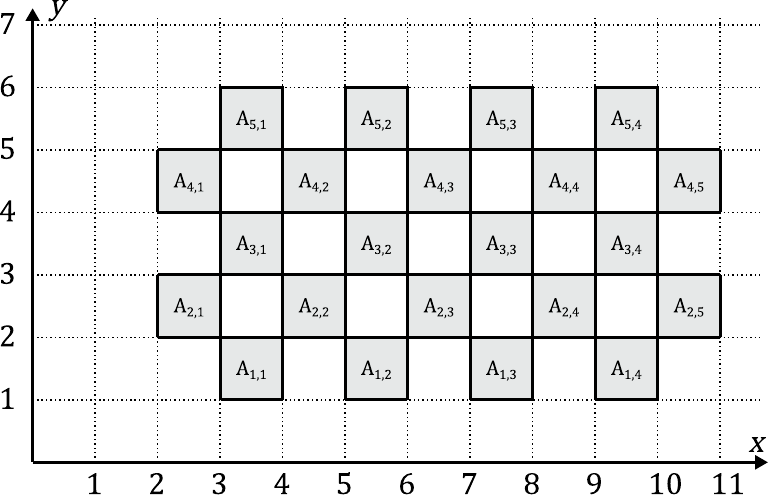}
    \caption{A \({(6,5)}\)-plat structure}
    \label{fig:190911PlatGrid}
\end{figure}

For each $y \in \mathbb{R},$ and $k \in \mathbb{Z}$, let $c_{y,k}$ be the circle of radius $\frac{1}{2}$ in $P_y$, centered at $x = k + \frac{1}{2}, z = 0.$ The \textit{plat tube} $A_{i,j}$ is defined to be the annulus
\begin{center}
    $A_{i,j} = \begin{cases} \lbrace c_{y,2j} \;\ | \;\ y \in [i,i+1]\rbrace & i \;\ \text{even} \\
\lbrace c_{y,2j+1} \;\ | \;\ y \in [i,i+1]\rbrace & i \;\ \text{odd}.
\end{cases}$

\end{center}

For $n,m \in \{2,3,4,\hdots\}$, the \textit{$(n,m)$-plat structure} is the union of the plat tubes $A_{i,j}$ where $i$ ranges from 1 to $n-1$ and $j$ either ranges from 1 to $m$ or 1 to $m-1$ depending upon whether $i$ is even or odd, respectively.

For $n,m \in \{2,3,4,\hdots\}$, an \textit{$(n,m)$-plat braid} is a union of $2m$ pairwise disjoint arcs in $\R^3$ whose projections to the $y$-axis are monotonic, satisfying the following properties:
\begin{enumerate}
    \item One endpoint of each arc lies in $P_1$ and the other endpoint lies in $P_n$.
    \item Each arc can be cut into subarcs, each of which is contained either in the \((n,m)\)-plat structure or in one of the vertical lines \(x=2\) or \(x=2m+1\) in the \(xy\)-plane.
    \item The intersection of the braid with each plat tube consists of a pair of arcs which intersect the plane $z=0$ in a minimal number of components and whose endpoints lie in $z=0.$
\end{enumerate}
Observe that the plane $z=0$ cuts $A_{i,j}$ into two disks.
Here we define the \textit{twist number} \(a_{i,j}\).
If the braid intersects $A_{i,j}$ in vertical arcs, then we define $a_{i,j}=0.$ 
Otherwise, the disk with non-negative $z$-coordinates contains some number of arcs of the plat braid whose projection to the plane $z=0$ is a set of parallel line segments.
We define $|a_{i,j}|$ to be this number of parallel arcs.
The sign of $a_{i,j}$ is defined to be the sign of the slope (\(\Delta y/\Delta x\))  of the line segments. 
The integer $a_{i,j}$ is called the \textit{twist number} for $A_{i,j}$.

For our purposes, we will only consider $(n,m)$-plat braids with $n$ even.
In this case, we can obtain a link from an $(n,m)$-plat braid by first connecting the point $(2j,1,0)$ to $(2j+1,1,0)$ for each $1 \leq j \leq m$ with the unique (up to isotopy) arc in the portion of the plane $z = 0$ which lies below the line $y=1$. Similarly, for each $1 \leq j \leq m$, we also connect the point $(2j,n,0)$ to the point $(2j+1,n,0)$ with the unique arc in the portion of plane $z = 0$ above the line $y=n$.
These \(2m\) arcs can be isotoped in the plane \(z=0\) (with respect to their endpoints) so that each is injective when projected to the \(x\)-axis and each contains either a single maximum or minimum point (with respect to \(h\)), with the result that the set of \(2m\) arcs is pairwise disjoint. The embedding of a link constructed as the union of the plat braid and these \(2m\) arcs in this way is said to be an \textit{\({(n,m)}\)-plat position of a link}.
If a link has an \({(n,m)}\)-plat position, it is called an \textit{\({(n,m)}\)-plat link}. A plat link is called \textit{\(k\)-twisted} if \(|a_{i,j}|\geq k\) for every twist number \(a_{i,j}\).

Throughout the rest of the paper, when discussing plat links, \(\alpha_+^1\hdots \alpha_+^m\) will refer to the bridge arcs above \(P_n\), labeled from left to right.
Likewise, let \(\alpha_-^1\hdots\alpha_-^m\) be the bridge arcs below \(P_1\), labeled from left to right.
Let \(L^i\) be the link component that contains \(\alpha_+^i\).
(Of course, since an \(m\)-bridge link may have fewer than \(m\) components, it may be that the component containing \(\alpha_+^i\) also contains \(\alpha_+^j\) for some \(j\neq i\), and so \(L^i=L^j\).)

There are two types of projection maps that we will often refer to. The first type of projection map is the Euclidean projection map $\sigma_i:\mathbb{R}^3 \rightarrow \mathbb{R}^3$ defined by $\sigma_i(x,y,z) = (x,i,z).$ The second type of projection map is the map $\pi_y:\mathbb{R} \times [1,n] \times \mathbb{R} \rightarrow P_y$, which sends each component of the plat braid to the corresponding point $(j,y,0),$ and extends to a homeomorphism from $P_{y'}$ to $P_y$ for each $y' \in [1,n].$ 
(In a slight abuse of notation, we will refer to this homeomorphism as \(\pi_y\).)

\begin{figure}
  \centering 
\includegraphics[width=1\textwidth]{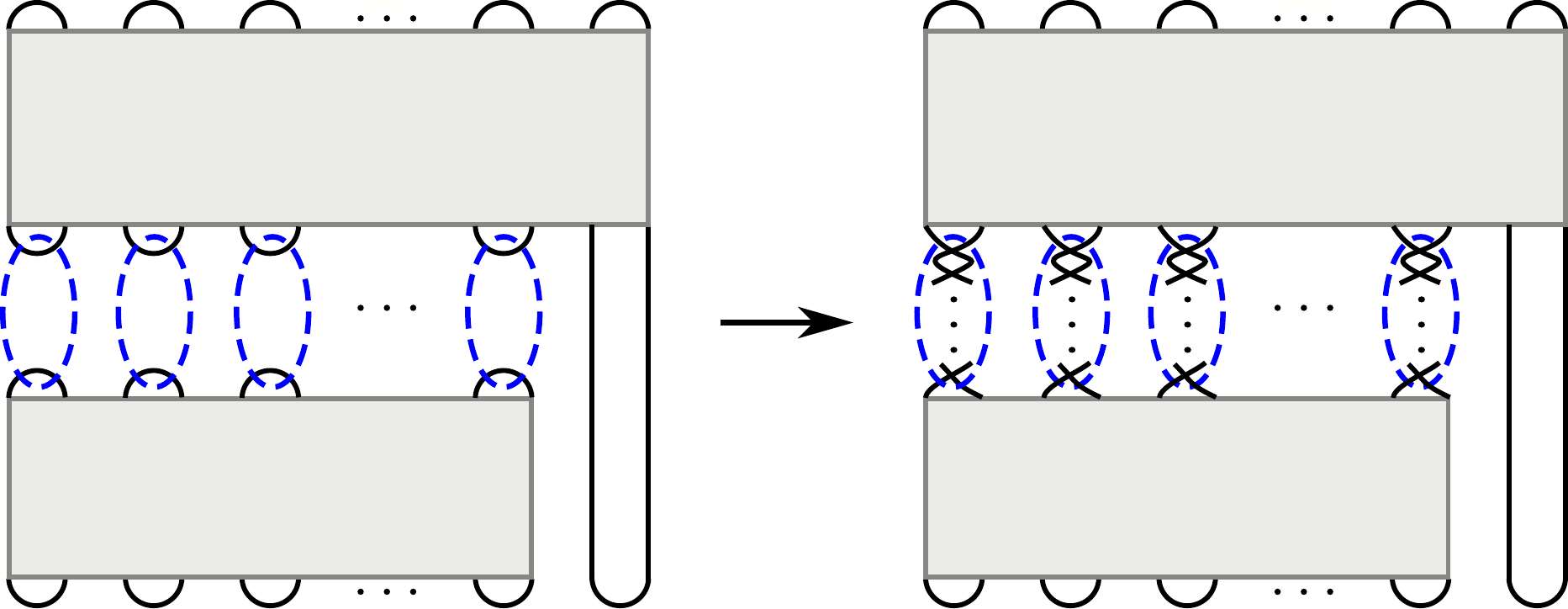}
\caption{The upper gray box contains \(D_a\), a 4-twisted \({(2m-4,m)}\)-plat braid, where \(m\geq 4\).
The lower gray box contains \(D_b\) a 4-twisted \({(n,m-1)}\)-plat braid. 
}
\label{fig:200320PlatTemplate}
\end{figure}

\begin{figure}
  \centering 
\includegraphics[width=.7\textwidth]{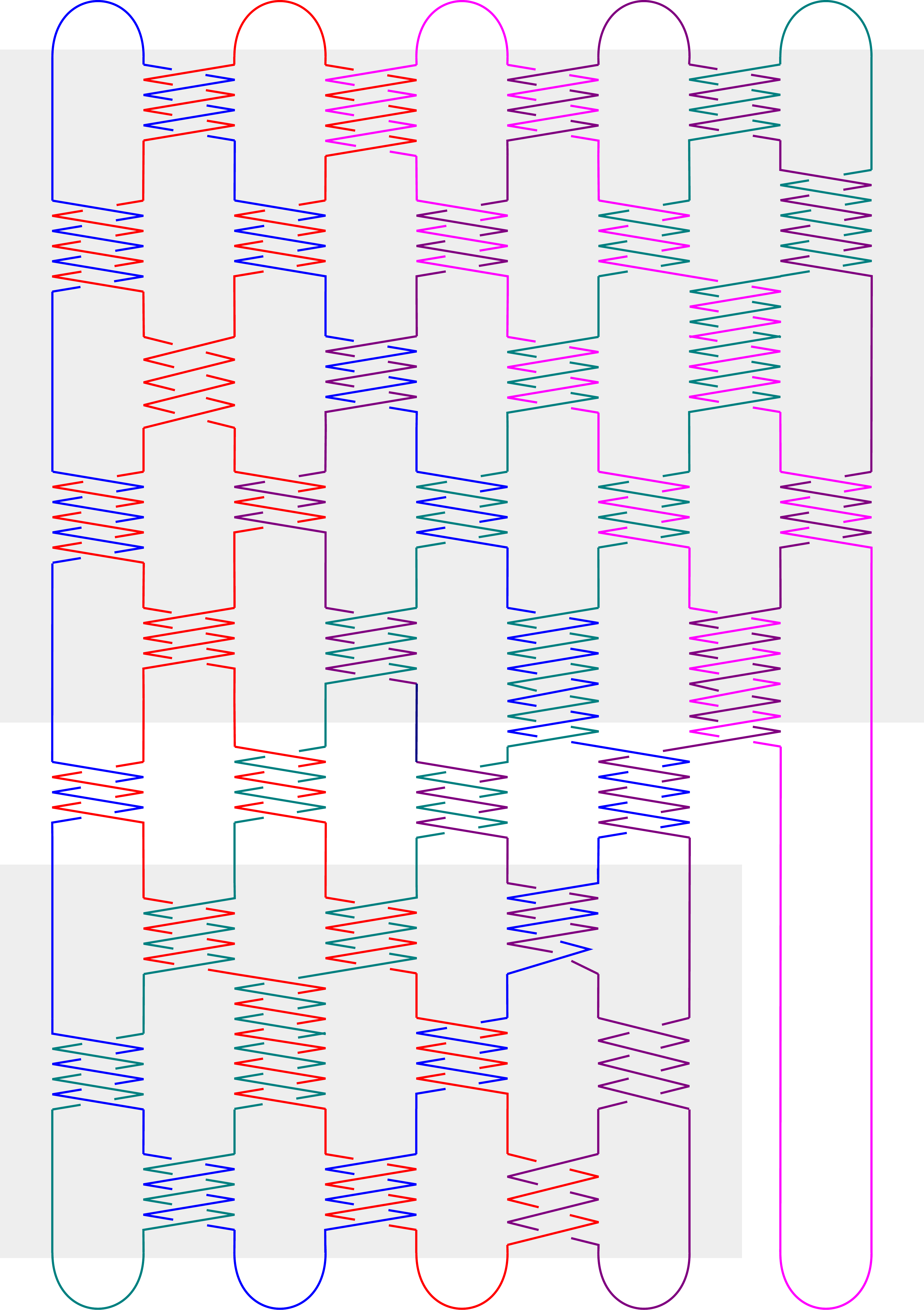}
\caption{A (10,5)-plat link $L \in \mathcal{L}.$}
\label{fig:200320PlatExampe6-5}
\end{figure}

\subsection{The Family of Links We Consider}\label{sec:family_of_links}

Let $D_a$ be a $4$-twisted \({(n_a,m)}\)-plat position of a link $L_a$ such that \(m\geq 4\) and  \(n_a=2m-4\), and let $D_b$ be a $4$-twisted \({(n_b,m-1)}\)-plat position of a link $L_b$. Position $D_a$ above $D_b$ as shown on the left of Figure \ref{fig:200320PlatTemplate}. Let $D_{ab}$ be an \({(n_a+n_b,m)}\)-plat position obtained from $D_a \sqcup D_b$ by replacing each of the 0-tangles in the dashed ovals with vertical half-twists as shown on the right of Figure \ref{fig:200320PlatTemplate}. 
(Note: The subscripts \(a\) and \(b\) are used for ``above'' and ``below.'')
We define \(\mathcal{L}\) to be the family of links constructed in this fashion which have the following additional properties.
\begin{enumerate}
\item\label{item:alternate-twisting} The rightmost twist regions of every row alternate in sign from row to row.
That is, the sign of the rightmost nonzero twist region of each row is opposite to the sign of the rightmost nonzero twist regions of any adjacent rows.
\item\label{item:twist-nos-for-L3} The sign of \(a_{n-1,2}\), the twist number for the second twist region in the top row of \(D_a\), is even, and every other twist region that involves \(L^3\) has an odd twist number.
(This forces \(L^3\) to be an unknot component containing the bridge arcs \(\alpha_+^3\) and \(\alpha_-^m\).)
\item\label{item:twist-nos-for-Lm}
The signs of the twist numbers for the twist regions involving \(L^m\) are chosen so that \(L^m\) contains the lower left bridge arc \(\alpha_-^1\).
\item\label{item:m-component} The rest of the twist numbers for \(D_{ab}\) are chosen so that \(D_{ab}\) is an \(m\)-component link and so that the bridges \(\alpha_+^m\) and \(\alpha_-^1\) are contained in the same link component, namely \(L^m\).
(It follows that for each \(i,j\in\{1,\hdots,m\}\), with \(i\neq j\), \(L^i\) is a distinct link component from \(L^j\).)
\item\label{item:nonsplit-pairs} Excluding the pair \(\{L^1,L^3\}\), every pair of link components comprises a two-bridge nonsplit sublink.
(Note: The sublink \(L^1\cup L^3\) will always be split no matter what set of twist numbers is chosen.)
\end{enumerate} 

Below in Proposition \ref{prop:family-exists}, we will show that \(\mathscr{L}\) is a nonempty set.
First, observe that $\mathcal{L}$ is a family of links in \({(n,m)}\)-plat position for $m \geq 4$ and \(n=n_a+n_b\) with the following conditions on the twist numbers:
\begin{enumerate}
\item For $i >n_b,$ $|a_{i,j}| \geq 4$ for all possible values of $j$.
 \item If $i$ is odd, $1 \leq i \leq n_b$, and $1 \leq j \leq m-2,$ (resp.\ $j = m-1$), then $|a_{i,j}| \geq 4$ (resp.\ $a_{i,j} = 0).$
 \item If $i$ is even, $1 \leq i \leq n_b$, and $1 \leq j \leq m-1,$ (resp.\ $j = m$), then $|a_{i,j}| \geq 4$ (resp.\ $a_{i,j} = 0).$
 \item If $a_{i,*}$ denotes the rightmost nonzero twist number in row $i$, then $a_{i,*}\cdot a_{i-1,*} < 0$ and $a_{i,*}\cdot a_{i+1,*} < 0$. In other words, the signs of the rightmost nonzero twist numbers alternate from row to row.

\end{enumerate}

Figure \ref{fig:200320PlatExampe6-5} shows an example of what $L \in \mathcal{L}$ may look like. In this case, $n=10,$ and $m=5.$
It follows from the definition of the family \(\mathscr{L}\) that for any \(L\in\mathscr{L}\), \(L^3\) is the component containing the lower right bridge arc \(\alpha_-^m\).

\begin{figure}
\centering 
\includegraphics[width=.8\textwidth]{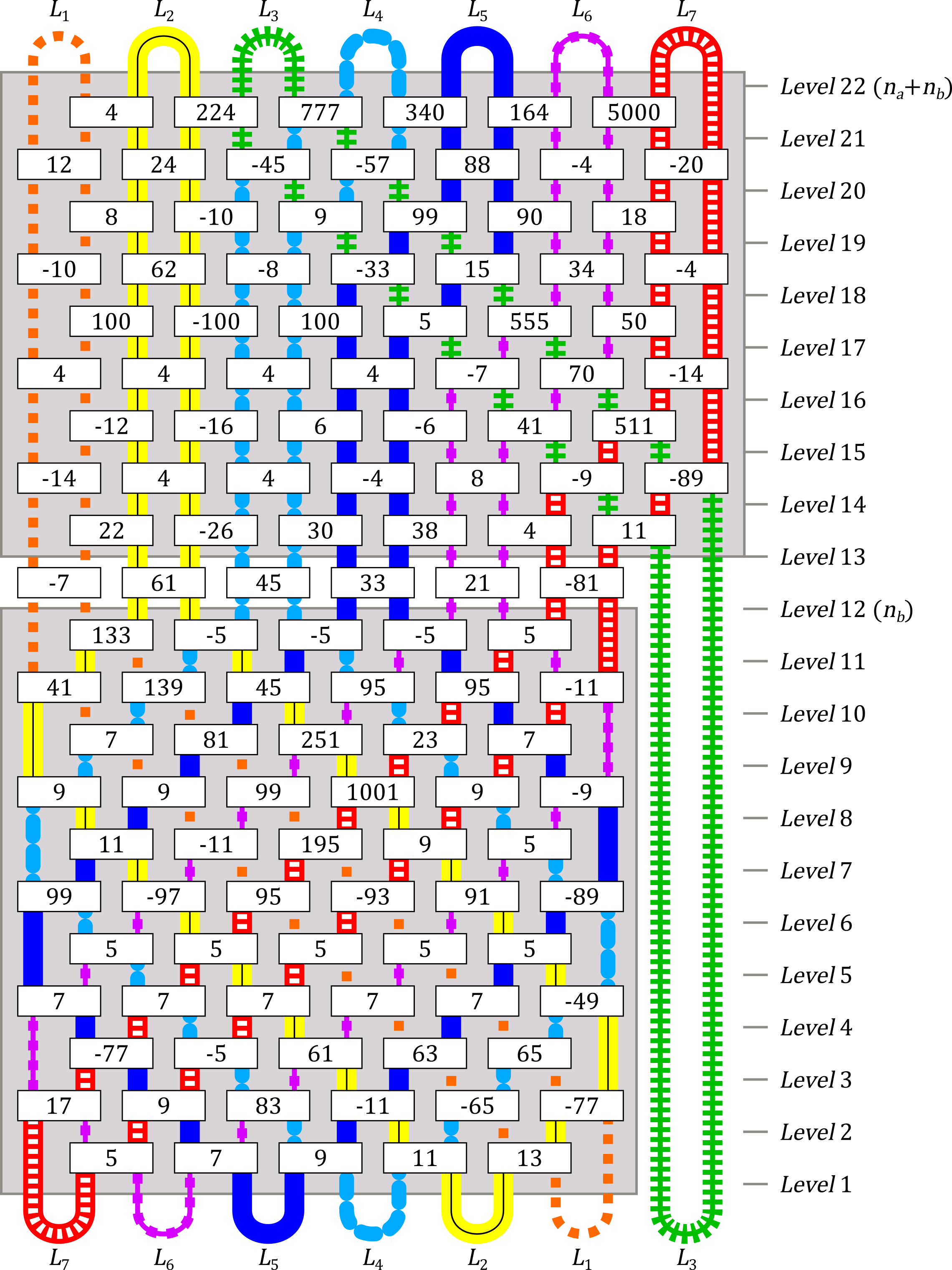}
\caption{This figure illustrates Proposition \ref{prop:family-exists}, showing an example a \(7\)-bridge link in \(\mathscr{L}\).
The seven different colors and line styles represent seven different link components. 
Each rectangle represents a twist region, and the integer inside each rectangle is the twist number indicating the number and sign of half twists present in that twist region.}
\label{fig:221008_Plat}
\end{figure}

\begin{prop}\label{prop:family-exists}
For each integer \(m\geq 4\), the family \(\mathscr{L}\) contains infinitely many links of bridge number \(m\).
\end{prop}

\begin{proof}
It is not difficult to satisfy properties \ref{item:alternate-twisting} and \ref{item:twist-nos-for-L3}.
We need to show that properties \ref{item:twist-nos-for-Lm}, \ref{item:m-component}, and \ref{item:nonsplit-pairs} can be satisfied.
We will do this by constructing an infinite family of examples for each bridge number \(m\geq 4\).

Fix \(m\geq 4\).
Let \(D_a\) and \(D_b\) be a \((2m-4,m)\)-plat link and an \((m-1, 2m-2)\)-plat link.
We will show that the right choices of the parities and magnitudes of the twist numbers of \(D_a\) and \(D_b\) will allow us to fulfill the conditions given above.

First, after choosing odd numbers for the particular twist numbers in \(D_a\) prescribed by property \ref{item:twist-nos-for-L3}, choose all even numbers for the rest of the twist numbers in \(D_a\).
Then choose any integers of magnitude at least four (regardless of parity) for the row of twist regions between \(D_a\) and \(D_b\).
The twist number choices we have made so far guarantee that at level \(n_b\), the punctures of \(P_{n_b}\) occur in pairs corresponding to the link components in this order, from left to right: \(L^1, L^2, L^4, L^5,\hdots,L^m,L^3\).
That is, they are arranged in numerical order from left to right except that the punctures of \(L^3\) appear at the end of the line.

Then for every twist region below \(P_{n_2}\) (i.e., the twist numbers corresponding to \(D_b\)), we choose all odd twist numbers.
This guarantees that the punctures of \(P_1\) occur in pairs corresponding from left to right to the link components \(L^m, L^{m-1},\hdots, L^5, L^4, L^2, L^1, L^3\).
That is, they area arranged in reverse numerical order, except that again, the punctures of \(L^3\) are at the end of the line.
Thus \(L\) is an \(m\)-component link whose lower left bridge arc is contained in \(L^m\), satisfying conditions \ref{item:twist-nos-for-Lm} and \ref{item:m-component}.

Since every twist number in \(D_b\) is odd, it follows that for each pair \(\{L^i,L^j\}\) of distinct link components from the set \(\{L^1, L^2, L^4,\hdots,L^m\}\) (the set of all link components excluding \(L^3\)), there are exactly four twist regions in \(D_b\) which involve both \(L^i\) and \(L^j\).
The choices of twist numbers in \(D_a\) guarantees that there is one twist region containing arcs of both \(L^2\) and \(L^3\), and there are exactly four twist regions containing arcs of both \(L^3\) and \(L^j\) for each \(j\geq 4\).
Now let \(\{L^i,L^j\}\) be any pair of link components except for the pair \(\{L^1,L^3\}\).
To satisfy condition \ref{item:nonsplit-pairs}, simply choose twist numbers so that the linking number of \(L^i\cup L^j\) is nonzero.
For example here is one way to do so.
There will be some positive number \(N\) of twist regions that involve strands from both \(L^i\) and \(L^j\).
For these twist regions, choose twist numbers \(t_1,t_2,\hdots,t_N\) so that \(|t_1|>\sum_{k=2}^N|t_k|\).
\end{proof}

\begin{prop}\label{prop:L-nonsplit}
Each link in the family \(\mathscr{L}\) is nonsplit.
\end{prop}

\begin{proof}
Let \(L\in\mathscr{L}\), and assume \(S\) is a splitting sphere for \(L\).

\textbf{Case 1:} \textit{The link components \(L^1\) and \(L^2\) are both on the same side of \(S\).}
In this case, let \(L^j\) be a link component on the other side of \(S\).
Then by condition \ref{item:nonsplit-pairs} of the definition of \(\mathscr{L}\), \(L^2\cup L^j\) is a nonsplit link which is split by \(S\), a contradiction.

\textbf{Case 2:} \textit{The link components \(L^1\) and \(L^2\) are not both on the same side of \(S\).}
Then \(S\) is a splitting sphere for the nonsplit link \(L^1\cup L^2\), another contradiction.
\end{proof}



To say that a given compressing disk \(C\) is a \textit{cap} is to say that there exists some bridge disk \(D\) such that \(C=dD\).
If \(\alpha\) is the bridge arc corresponding to \(D\), then we say that \(C\) is a cap \textit{for} \(\alpha\).
It follows that \(\partial C\) cuts the bridge sphere into two components, one of which is a twice-punctured disk, where the two punctures are the intersection points of the bridge sphere with \(\alpha\).

\begin{prop}\label{prop:corner-caps}
Let \(D\) and \(E\) be compressing disks above and below \(P\), respectively.
If \(\{D,E\}\) is a weak reducing pair for \(P\), then \(D\) is a cap for \(\alpha_+^1\), and \(E\) is a cap for \(\alpha_-^m\).
\end{prop}

\begin{proof}
The loop \(\partial D\subset P\) partitions the link components of \(L\) into two nonempty sets, \(A\) and \(A'\) (based on which side of \(\partial D\) the punctures of each link component lie).
Let \(A\) be the set containing \(L^1\).
The loop \(\partial E\) also partitions the link components into two nonempty sets, \(B\) and \(B'\).
Let \(B\) be the set containing \(L^3\).
If \(A\) contains \(L^i\) for any \(i\neq 1\), then \(\{D,E\}\) is a weak reducing pair for the sublink \(L^i\cup L^3\), a nonsplit 2-bridge link, a contradiction.
Similarly, if \(B\) contains \(L^j\) for any \(j\neq 3\), then \(\{D,E\}\) is a weak reducing pair for the sublink \(L^1\cup L^j\), a nonsplit 2-bridge link, another contradiction.
Therefore \(D\) is a cap for \(\alpha_+^1\), the bridge arc above \(P\) contained in \(L^1\), and \(E\) is a cap for \(\alpha_-^m\), the bridge arc below \(P\) contained in \(L^3\).
\end{proof}

The rest of the paper will be devoted to proving that each $L \in \mathcal{L}$ admits a keen weakly reducible bridge sphere. 
The reason Proposition \ref{prop:corner-caps} does not immediately imply this is because for any given bridge arc, there are infinitely many distinct caps for that bridge arc, provided there are at least three bridges on each side of the bridge sphere, which is the case for all of the links in \(\mathscr{L}\).

\begin{figure}
      \centering 
    \includegraphics[width=.2\textwidth]{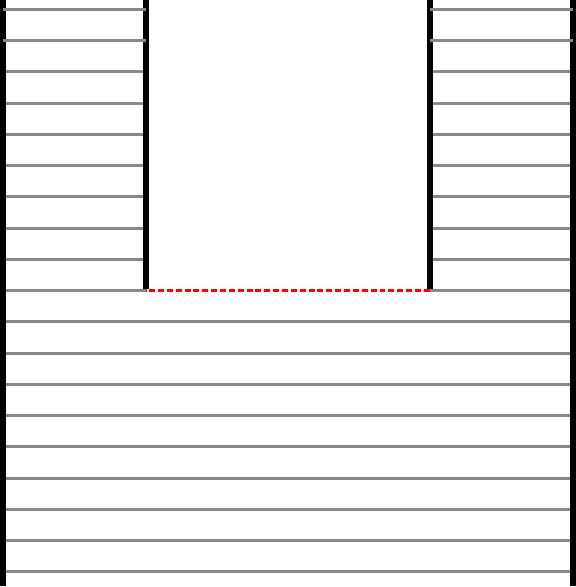}
    \caption{A train track at a singular fiber. 
    The closed red line segment is a switch.}
    \label{fig:200410SingularFiber}
\end{figure}

\subsection{Plat Train Tracks}\label{subsec:Plat_train_tracks}
Speaking generally, let $\Sigma_L$ denote a bridge sphere, and let $I$ denote a closed unit interval. A \textit{train track} $\tau$ is a compact subsurface of $\Sigma_L$ whose interior is fibered by open intervals and the fibration extends to a fibration of $\tau$ by closed intervals except for at finitely many intervals called \textit{singular fibers}. Let $\alpha$ be a singular fiber, and denote its closed neighborhood in $\tau$ by $\overline{N}(\alpha)$. Then, there is a homeomorphism $f:\overline{N}(\alpha) \ra (I \times I) \backslash \left(\left(\frac{1}{4},\frac{3}{4}\right)\times \left(\frac{1}{2},1\right]\right)$ such that $f(\alpha)=I\times \left\lbrace \frac{1}{2} \right\rbrace$. We will refer to the inverse image of $(I\times \left\lbrace \frac{1}{2} \right\rbrace) \backslash \left(\left(I \times \left[0,\frac{1}{4}\right]\right) \sqcup \left(I \times \left[\frac{3}{4},1\right]\right)\right)$ under $f$ as a \textit{switch} of $\tau$.
See Figure \ref{fig:200410SingularFiber}.

In this paper, we will assign a train track $\tau_i$ to each bridge sphere $P_i$ for $i = 1,2,\dots,n-1$. 
(There is no need for a train track at the top level \(P_n\).)
To this end, we will construct a certain trivalent graph, called a train graph, on each bridge sphere based on the parity of $i$ and the twist numbers $a_{i,j}$ for the row.
The train track will then be constructed from the train graph in a natural way.

We define a \textit{train graph} to be a connected trivalent graph with the property that the three edges incident to each vertex are tangent to each other at the vertex, and not all three edges emanate from the vertex in the same direction.
(See the left side of Figure \ref{fig:200417_TrainTrackConstruction}.) 
Below, we will construct a specific train graph \(T_i\) embedded in \(P_i\) for each \(i=1,2,\hdots,n-1\), and these train graphs will have the property that \(P_i\backslash T_i\) consists of \(2m\) once-punctured disks and one (non-punctured) disk.
We will informally express this by saying that each puncture is ``surrounded by'' \(T_i\).

To construct each train graph, there are various cases to consider. Recall from Subsection \ref{sec:family_of_links} that $\mathcal{L}$ is a family of links in \({(n,m)}\)-plat position for \(n=n_a+n_b\).
If $i$ is odd and $n-1 \geq i \geq n_b+1$ (resp. $i < n_b+1$), we define $\ell_{i,j}$ to be the circle in $P_i$ centered at $\left(2j+\frac{3}{2},i,0\right)$ with radius $\frac{3}{4}$ for $j = 1,2,\dots,m-1$ (resp. for $j=1,2,\dots,m-2).$ 
If $i$ is even and $n-2 \geq i \geq n_b+1$ (resp. $i < n_b+1)$, we define $\ell_{i,j}$ to be the circle in $P_i$ centered at $\left(2j+\frac{1}{2},i,0\right)$ with radius $\frac{3}{4}$ for $j = 1,2,\dots,m$ (resp. for $j=1,2,\dots,m-1).$

\begin{figure}
      \centering
    \includegraphics[width=.4\textwidth]{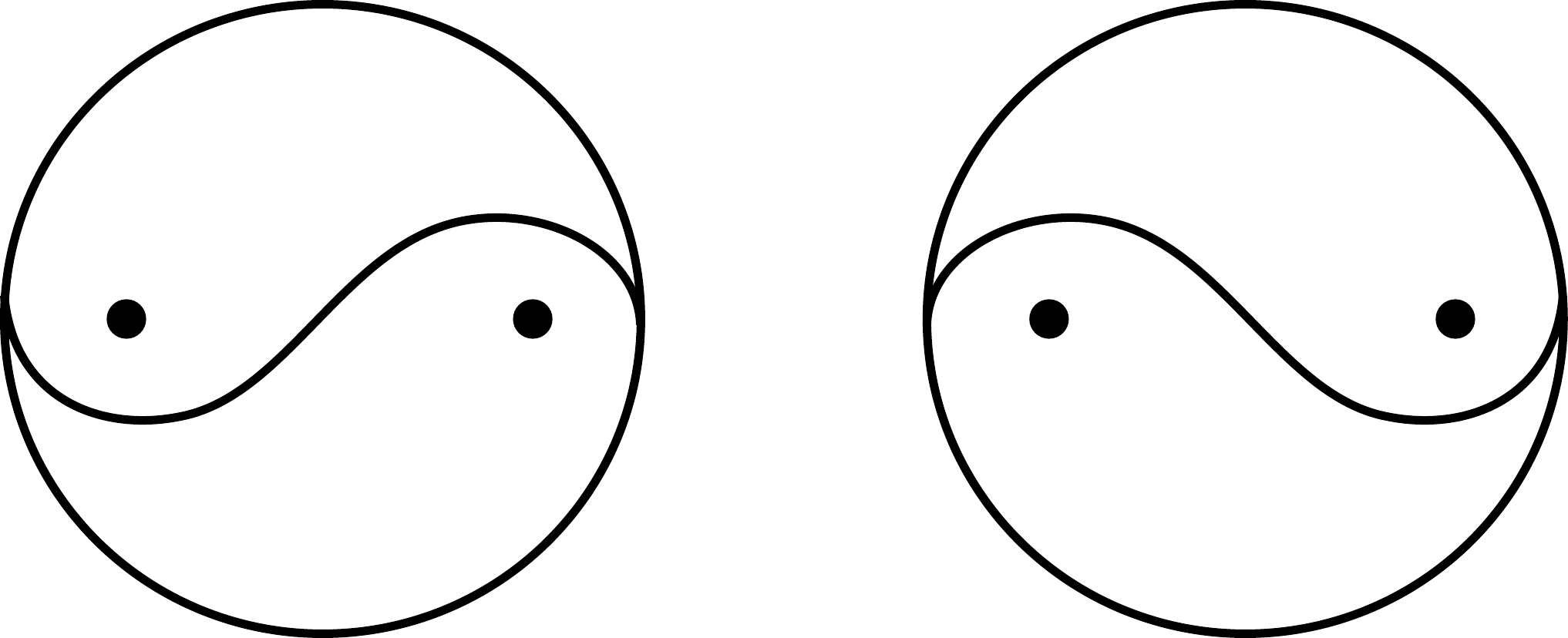}
    \caption{At left, a left-handed tao diagram. At right, a right-handed tao diagram.}
    \label{fig:taonobands}
\end{figure}

Now, each $\ell_{i,j}$ cuts out a twice-punctured disk from $P_i.$ We will distinguish two types of arcs that separate the two punctures. 
If $\ell_{i,j}$ is directly below a positive twist region, then we draw a \textit{right-handed tao arc} separating the two punctures as shown on the right of Figure \ref{fig:taonobands}. 
In the case where $\ell_{i,j}$ is directly below a negative twist region, we instead draw a \textit{left-handed tao arc}.
The union of a left-handed tao arc (resp.\ right-handed tao arc) with $\ell_{i,j}$ will be called a \textit{left-handed tao diagram} (resp.\ \textit{right-handed tao diagram}).
An important aspect of these tao diagrams is that at a tao arc's endpoints, the circle and the tao arc are tangent to each other as pictured.

\begin{figure}
    \centering
    \includegraphics[width=.7\textwidth]{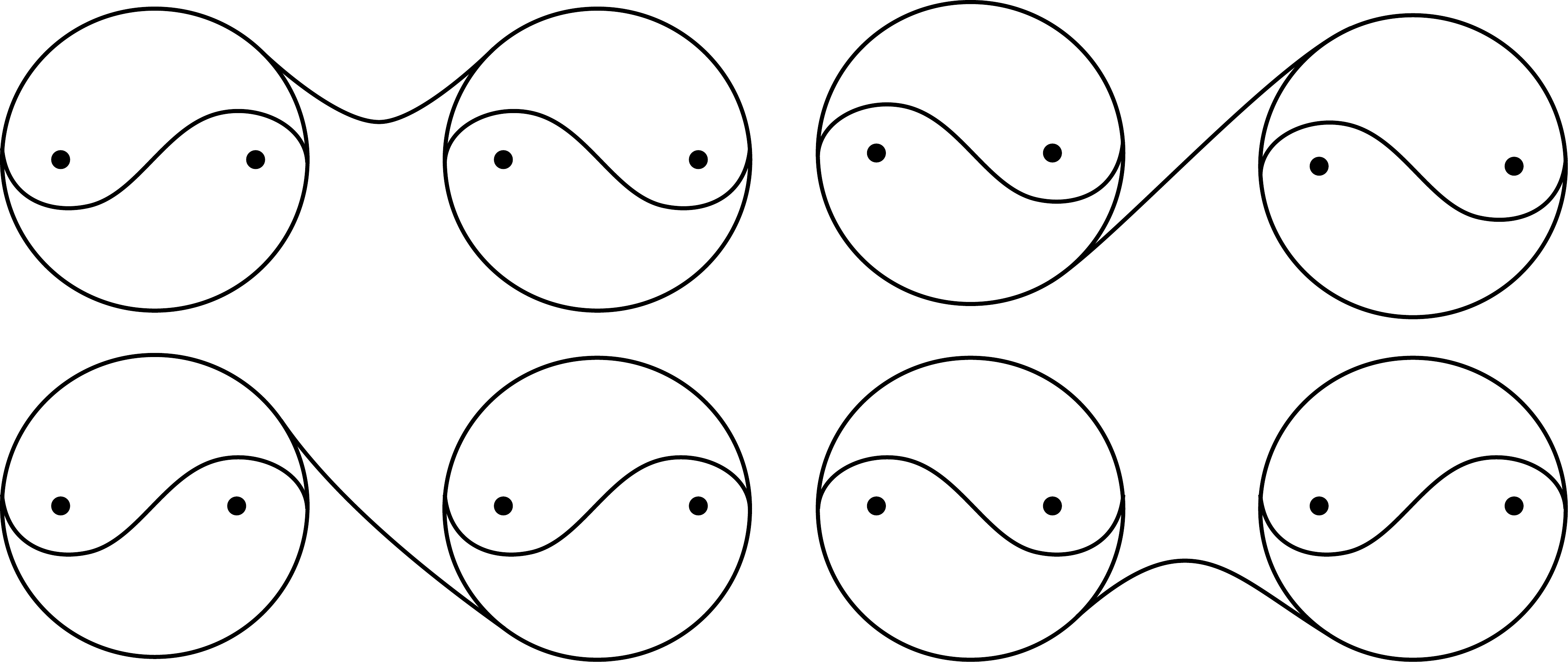}
    \caption{The way we add an edge between two adjacent tao diagrams depends on their handedness.}
    \label{fig:connecttao}
\end{figure}

At this point we have constructed various disconnected tao diagrams in each $P_i$. 
We next connect each pair of adjacent tao diagrams with an edge in one of the four ways pictured in Figure \ref{fig:connecttao}, depending on the handedness of each tao diagram.
If $i$ is even and $i\geq n_b+1,$ the result of this procedure is a train graph which we call $T_i$.

\begin{figure}
    \centering
    \includegraphics[width=.5\textwidth]{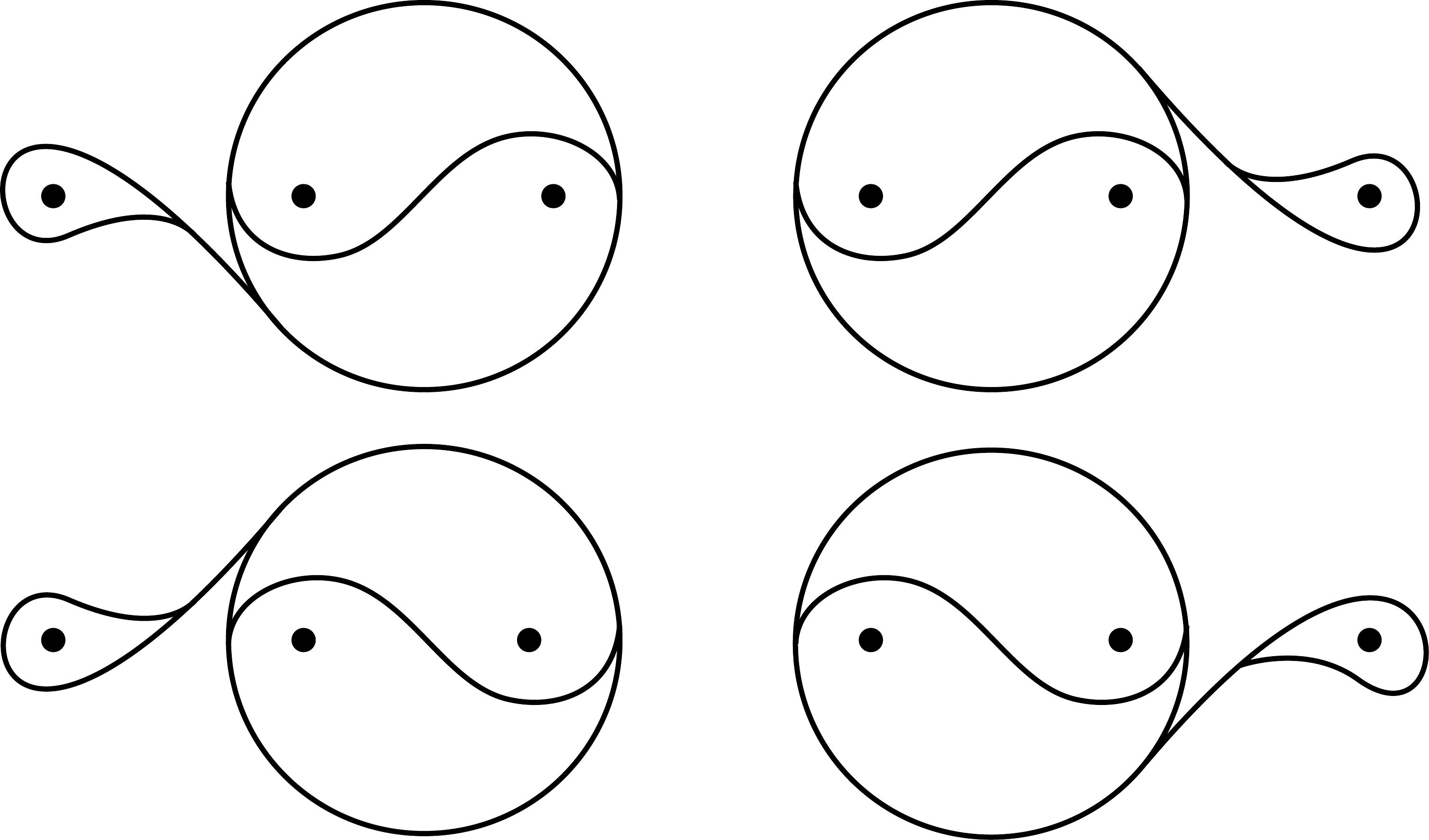}
    \caption{Adding eyelets to ``leftover" punctures that are adjacent to a tao diagram}
    \label{fig:eyelet1}
\end{figure}

In all other cases (i.e., if $i$ is odd and/or $i<n_b+1$), begin the construction of the train graph as above, combining tao diagrams and connecting edges; however, after doing so, there will be ``leftover'' punctures that are not surrounded by any tao diagrams. If any such puncture is adjacent to a puncture surrounded by a tao diagram, then we modify our graph according to Figure \ref{fig:eyelet1}, adding a vertex and two edges to the graph in a way that depends on which side of the tao diagram the puncture is on and the handedness of the tao diagram.
The newly added subgraph consists of two edges, one forming a loop around a puncture, and the other connecting the loop to a tao diagram. 
We refer to such a subgraph as an \textit{eyelet}.
If $i$ is odd and $i\geq n_b+1,$ this procedure gives a connected trivalent graph containing two eyelets, surrounding all the punctures. We call this train graph \(T_i\).

\begin{figure}[ht!]
\labellist
\tiny\hair 2pt
\pinlabel $P_{n-(2m-4)}$ at -180 1106
\pinlabel $P_{n-(2m-3)}$ at -180 874
\pinlabel $P_{n-(2m-2)}$ at -180 640
\pinlabel $P_{n-(2m-1)}$ at -180 406
\endlabellist

    \centering
    \includegraphics[width=.7\textwidth]{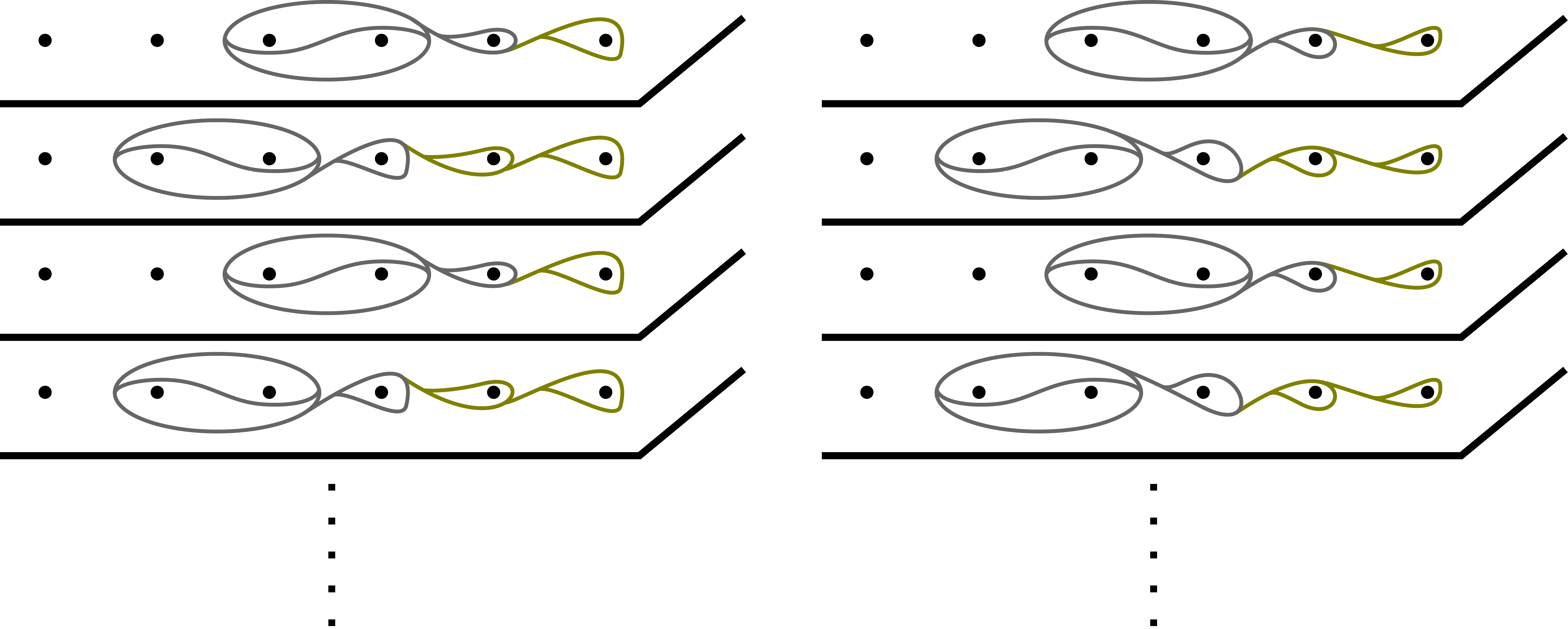}
    \caption{Adding eyelets to ``leftover'' punctures on $P_i$ for $1\leq i <  n_b+1$. If the rightmost tao in $P_{n-(2m-4)}$ is left-handed (resp. right-handed), we add eyelets according to the left (resp. right) picture.}
    \label{fig:eyelet2}
\end{figure}

For $1\leq i <  n_b+1,$ there are still ``leftover'' punctures that are not surrounded by a tao diagram or eyelet. Since the sign of the rightmost nonzero twist region of a row is opposite to the sign of the rightmost nonzero twist region of adjacent rows, there are only two possibilities for the rightmost tao and the eyelet adjacent to it. These are depicted in gray in Figure \ref{fig:eyelet2}. If $a_{n-1,m-1} < 0$, we add the eyelets according to the left of Figure \ref{fig:eyelet2}. If $a_{n-1,m-1} > 0$, we add the eyelets according to the right of Figure \ref{fig:eyelet2}. The newly added eyelets are colored gold. After doing so, we have a train graph that surrounds all the punctures for each $P_i$ for $1\leq i \leq n-(2m-4)$, and we call this train graph \(T_i\).

We have constructed a train graph \(T_i\) on each bridge sphere \(P_i\).
Now we will use each train graph \(T_i\) to construct a train track \(\tau_i\) on each sphere \(P_i\).
Let \(V_i\) and \(E_i\) be the vertex set and the edge set for \(T_i\), respectively.
For each vertex \(v\in V_i\), let \(\overline{N}(v)\) be a closed regular neighborhood of \(v\) in \(P_i\).

Let \(e'\) denote the connected component of \(T_i\backslash \bigcup_{v\in V_i}\overline{N}(v)\) corresponding to the edge \(e\).
Let \(\overline{N'}(e')\) be a closed regular neighborhood of \(e'\) in \(P_i\), and then define \(\overline{N}(e)=\overline{N'}(e')\backslash\bigcup_{v\in V_i}\overline{N}(v)\).
Notice that \(\left(\bigcup_{v\in V_i}\overline{N}(v)\right)\sqcup\left(\bigcup_{e\in E_i}\overline{N}(e)\right)\) is a regular neighborhood of \(T_i\) which we call \(\overline{N}(T_i)\), and the set \(\left\{\overline{N}(e), \overline{N}(v)\mid e\in E_i,v\in V_i\right\}\) is a partition for \(\overline{N}(T_i)\).

\begin{figure}
    \centering
    \includegraphics[width=.25\textwidth]{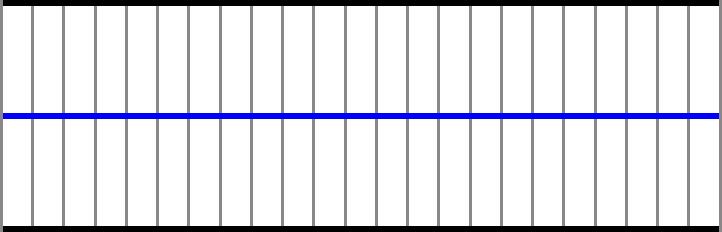}
    \caption{A fibration of $\overline{N}(e)$ by intervals.}
    \label{fig:200417_Nhood_of_e-prime}
\end{figure}

\begin{figure}
    \centering
    \includegraphics[width=.25\textwidth]{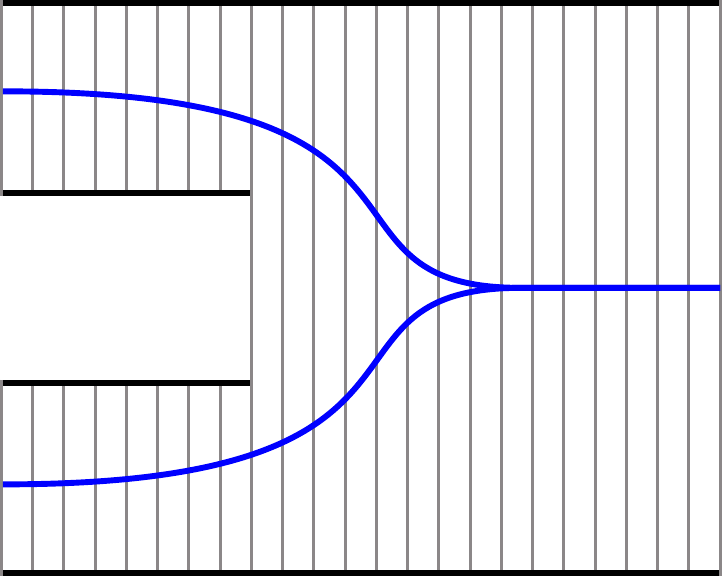}
    \caption{A singular fibration on \(\overline{N}(v)\) containing one singular fiber.}
    \label{fig:200417_Nhood_of_v-prime}
\end{figure}

\begin{figure}
    \centering
    \includegraphics[width=1\textwidth]{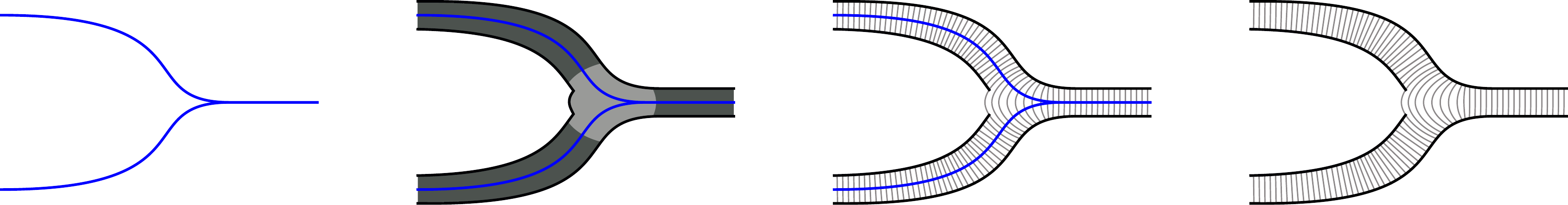}
    \caption{Constructing a train track \(\tau_i\) from the train graph \(T_i\).}
    \label{fig:200417_TrainTrackConstruction}
\end{figure}

We fiber each set \(\overline{N}(e)\) with interval fibers, each one intersecting \(e'\) transversely exactly once as in Figure \ref{fig:200417_Nhood_of_e-prime}.  
Then we impose a singular fibration on each \(\overline{N}(v)\) containing exactly one singular fiber, as in Figure \ref{fig:200417_Nhood_of_v-prime}.
This makes \(\overline{N}(v)\) into a neighborhood of a switch in a train track.
The surface \(\overline{N}(T_i)\), together with the singular fibration, is a train track which we call \(\tau_i\), constructed from the train graph \(T_i\).
This construction process is illustrated in Figure \ref{fig:200417_TrainTrackConstruction}.

\subsection{Carried and Almost Carried}\label{subsec:carried_almost_carried}

We want to isotope certain objects in the bridge sphere \(P_i\) to a position that behaves nicely with respect to the train track \(\tau_i\).

\begin{figure}
    \centering
    \includegraphics[width=.3\textwidth]{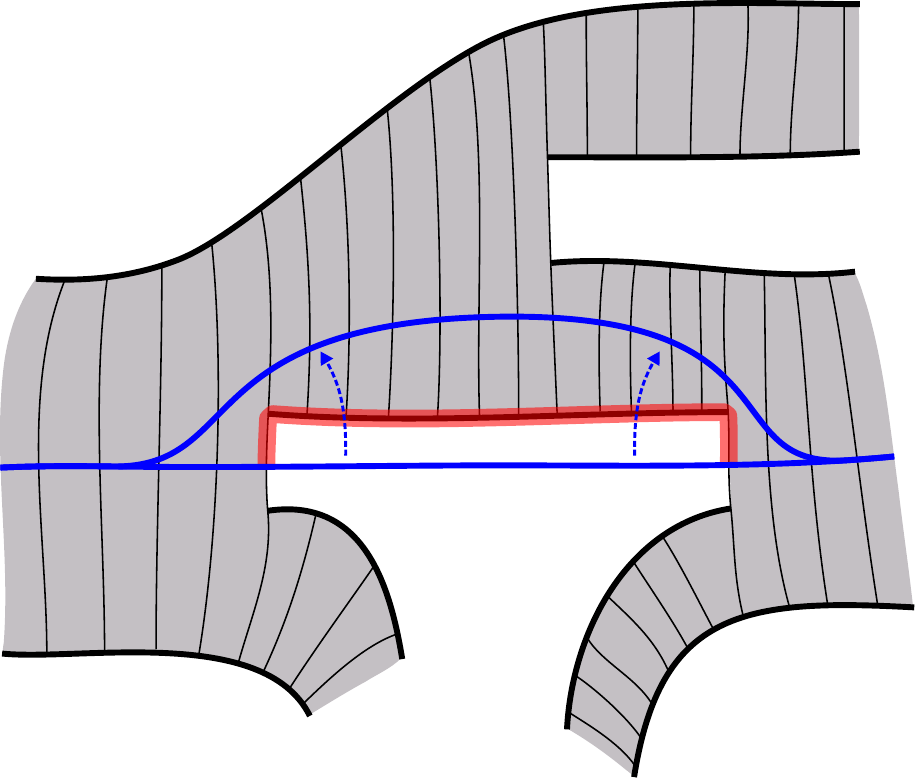}
    \caption{The red shaded arc is \(\alpha''\) from Definition \ref{defn:almostcarried_arc}.}
    \label{fig:200501_almost_carried}
\end{figure}

\begin{defn}\label{defn:almostcarried_arc}
For an arc $\alpha$ (not necessarily properly) embedded in $P_i$, \(\tau_i\) is said to \textit{almost carry} \(\alpha\) if the following are true.
\begin{enumerate}
\item\label{def:almost_carried_condition_1} For each point \(p\in\alpha\), either \(p\not\in\tau_i\) or \(p\) is a transverse intersection point of \(\alpha\) with an interval fiber of \(\tau_i\).
\item\label{def:almost_carried_condition_2} 

No point of \(\alpha\) is an endpoint of an interval fiber of \(\tau_i\).
\item\label{def:almost_carried_condition_3}

No connected component \(\alpha'\) of \(\alpha\backslash\interior{\tau_i}\)  is parallel (rel \(\partial\alpha'\)) into a switch.

\item\label{def:almost_carried_condition_4}
No connected component \(\alpha'\) of \(\alpha\backslash\interior{\tau_i}\)  is parallel (rel \(\partial\alpha'\)) into an arc \(\alpha''\subseteq\partial\tau_i\) with the property that \(\alpha''\) is partitioned into three subintervals: the outer two being subintervals of switches and the middle subinterval of \(\alpha''\) being an interval of fiber endpoints of \(\tau_i\).
(See Figure \ref{fig:200501_almost_carried}.)

\end{enumerate}
\end{defn}

\begin{rem}
It follows from Definition \ref{defn:almostcarried_arc} that if \(\tau_i\) almost carries an arc \(\alpha\), and an endpoint of \(\alpha\) lies in \(\partial\tau_i\), then that endpoint lies in the interior of a switch.
\end{rem}

\begin{rem}\label{rem:fulfill_conditions_1-2}
If an arc \(\alpha\) satisfies conditions \eqref{def:almost_carried_condition_1}, \eqref{def:almost_carried_condition_2}, and \eqref{def:almost_carried_condition_3} of Definition \ref{defn:almostcarried_arc}, then each arc of \(\alpha \cap (P_i\backslash\interior{\tau_i})\) which is properly embedded in \(P_i\backslash\interior{\tau_i}\) but which does not satisfy condition \eqref{def:almost_carried_condition_4} can be isotoped into the train track, as illustrated in Figure \ref{fig:200501_almost_carried}.
This results in a position of \(\alpha\) which \textit{is} now almost carried.
\end{rem}

\begin{defn}
A loop \(\ell\subseteq P_i\) is said to be \textit{almost carried} by the train track \(\tau_i\) if every connected component of \(\tau_i\cap\ell\) and every connected component of \(\left(P_i\backslash\interior{\tau_i}\right)\cap\ell\) is an arc which is almost carried by \(\tau_i\).

\end{defn}

\begin{rem}\label{rem:disjoint}
An arc or loop in \(P_i\) which is completely disjoint from \(\tau_i\) still satisfies the definition of being almost carried by \(\tau_i\).
\end{rem}

\begin{defn}\label{def:almost-carried-train-graph}
A train graph \(T\) is said to be \textit{almost carried} by the train track \(\tau_i\) if each edge of \(T\) is almost carried by \(\tau_i\).
\end{defn}

As a simple example, for each \(i\), the train graph \(T_i\) is almost carried by the train track \(\tau_i\).

\begin{defn}\label{def:covers}
Let \(T_i'\) be a subgraph of the train graph \(T_i\) (e.g., a tao), and let \(\tau_i'\subseteq\tau_i\) be the sub train track constructed from $T_i'$ following the instruction in Subsection \ref{subsec:Plat_train_tracks}.
If \(\ell\) is a loop or train graph, then \(\ell\) is said to \textit{cover} \(T_i'\) if \(\ell\) is almost carried by \(\tau_i'\) and \(\ell\) intersects every interval fiber of \(\tau_i'\).
\end{defn}

\begin{defn}\label{defn:almost_carry_train_track_diagrams}
Let \(\tau\) and \(\sigma\) be two different train tracks contained in the same bridge sphere \(P_i\).
Then \(\sigma\) is said to \textit{almost carry} \(\tau\) if for each interval fiber \(I\) of \(\tau\), \(I\) is disjoint from \(\sigma\) or \(I\) is contained in the interior of some interval fiber of \(\sigma\). 
\end{defn}

\begin{figure}
\centering
\includegraphics[width=.65\textwidth]{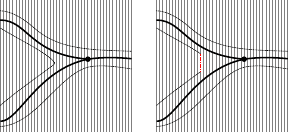}
\caption{The black graph is \(T\), the vertical gray lines are fibers of \(\sigma\), and \(\overline{N}(T)\), the closed regular neighborhood of  \(T\), is outlined with dashed lines.
On the right, we see that a slight isotopy of \(\overline{N}(T)\) makes it into a train track diagram \(\tau\) with fibers inherited from \(\sigma\).
That is, every fiber of \(\tau\) is a subinterval of a fiber of \(\sigma\).}
\label{fig:200713_reg_nhood_of_train_graph}
\end{figure}

\begin{prop}\label{prop:train-track-in-train-track}
Let \(T\) be a train graph in the bridge sphere \(P\), corresponding to train track diagram \(\tau\), and let \(\sigma\) be another train track diagram in \(P\).
If \(T\) is almost carried by \(\sigma\), then \(\tau\) is almost carried by \(\sigma\).
Furthermore, if \(p\) is a point in \(T\), and \(p\) lies in the interval fiber \(I\subseteq\tau\), and \(p\) lies in the interval fiber \(J\subseteq\sigma\), then \(I\subseteq J\).
\end{prop}

\begin{proof}
Our strategy here is to reexamine the construction of \(\tau\) and see that it has the desired properties.
Let \(T\) be a train graph in \(P\), almost carried by \(\sigma\), and let \(\overline{N}(T)\) be a closed regular neighborhood of \(T\).

By definition, every point of \(T\) is either disjoint from \(\sigma\) or lies in the interior of a fiber interval of \(\sigma\).
It follows that \(\overline{N}(T)\) is disjoint from the fiber endpoints of \(\sigma\).

Near each vertex \(v\) of \(T\) which lies in \(\sigma\), we perform a slight isotopy of \(\overline{N}(T)\) (pictured in Figure \ref{fig:200713_reg_nhood_of_train_graph}) as follows:
We locate an arc \(\lambda\) of \(\partial \overline{N}(T)\) located between the two edges of \(T\) which emanate from \(v\) in the same direction.
We isotope \(\overline{N}(T)\) so that \(\lambda\) is a subinterval of one of the fiber intervals of \(\sigma\).

Next we allow the portion of \(\overline{N}(T)\) which lies inside \(\sigma\) to inherit a fibering from \(\sigma\) this way:
If \(J\) is a fiber of \(\sigma\), then \(J \cap\overline{N}(T)\) is a (possibly empty) set of fibers of \(\overline{N}(T)\).

After extending this fibration to the rest of \(\overline{N}(T)\) which lies outside of \(\sigma\), \(\overline{N}(T)\), endowed with a fibration, is now a train track diagram \(\tau\) with the desired properties.
\end{proof}

\section{How Compressing Disks Meet Train Tracks}
It is desirable to isotope a simple closed curve to intersect a train track in a way over which we can have some control. 

For future convenience, we partition the compressing disks into two disjoint sets. Consider $\alpha_+^1$, the left most bridge arc above \(P_n\).
A vertical isotopy of $\alpha_+^1$ into bridge sphere $P_n$ traces out a bridge disk $D^1_+$. Let $B=dD^1_+$; that is, \(B\) is the cap which is the frontier of a regular neighborhood of $D^1_+$ in $V_+$.
Similarly, consider $\alpha_-^{m}$, the rightmost bridge arc below \(P_1\).
The bridge arc \(\alpha_-^m\) gives rise to a bridge disk $D^{m}_-$ and a corresponding caps $B'=dD^m_-$. We will refer to these two isotopy classes of caps as \textit{blue} disks. Compressing disks for $P_i$ that are not blue will be referred to as \textit{red} disks.

\begin{figure}
\centering
\labellist
\tiny\hair 2pt
\pinlabel  {$\gamma$}  at 48 395
\pinlabel  {$r$}  at 248 276
\pinlabel  {$r$}  at 460 276
\pinlabel  {$r$}  at 670 276
\pinlabel  {$r$}  at 248 200
\pinlabel  {$r$}  at 460 200
\pinlabel  {$r$}  at 670 200
\pinlabel  {\fontsize{5}{6}\selectfont{$(k-1)\cdot r$}}  at 195 253
\pinlabel  {\fontsize{5}{6}\selectfont{$k\cdot r$}}  at 282 253
\pinlabel  {\fontsize{5}{6}\selectfont{$k\cdot r$}}  at 199 140
\pinlabel  {\fontsize{5}{6}\selectfont{$(k-1)\cdot r$}}  at 280 133
\pinlabel  {$r$}  at 248 116
\pinlabel  {$r$}  at 460 116
\pinlabel  {$r$}  at 655 116
\pinlabel  {$r$}  at 40 256
\pinlabel 
{\fontsize{5}{6}\selectfont{$(k-2)\cdot r$}}  at 435 236
\pinlabel  {\fontsize{5}{6}\selectfont{$(k-1)\cdot r$}}  at 481 223
\pinlabel  {\fontsize{5}{6}\selectfont{$(k-1)\cdot r$}}  at 422 172
\pinlabel  {\fontsize{5}{6}\selectfont{$(k-2)\cdot r$}}  at 470 153
\pinlabel 
{\fontsize{5}{6}\selectfont{$(k-3)\cdot r$}}  at 646 236
\pinlabel  {\fontsize{5}{6}\selectfont{$(k-1)\cdot r$}}  at 692 223
\pinlabel  {\fontsize{5}{6}\selectfont{$(k-2)\cdot r$}}  at 665 147
\pinlabel  {\fontsize{5}{6}\selectfont{$r$}}  at 290 27
\pinlabel  {\fontsize{5}{6}\selectfont{$r$}}  at 400 5
\pinlabel  {\fontsize{5}{6}\selectfont{$s$}}  at 277 73
\pinlabel  {\fontsize{5}{6}\selectfont{$t$}}  at 312 73
\pinlabel  {\fontsize{5}{6}\selectfont{$s$}}  at 372 86
\pinlabel  {\fontsize{5}{6}\selectfont{$t$}}  at 429 86
\endlabellist
\includegraphics[width=1\textwidth]{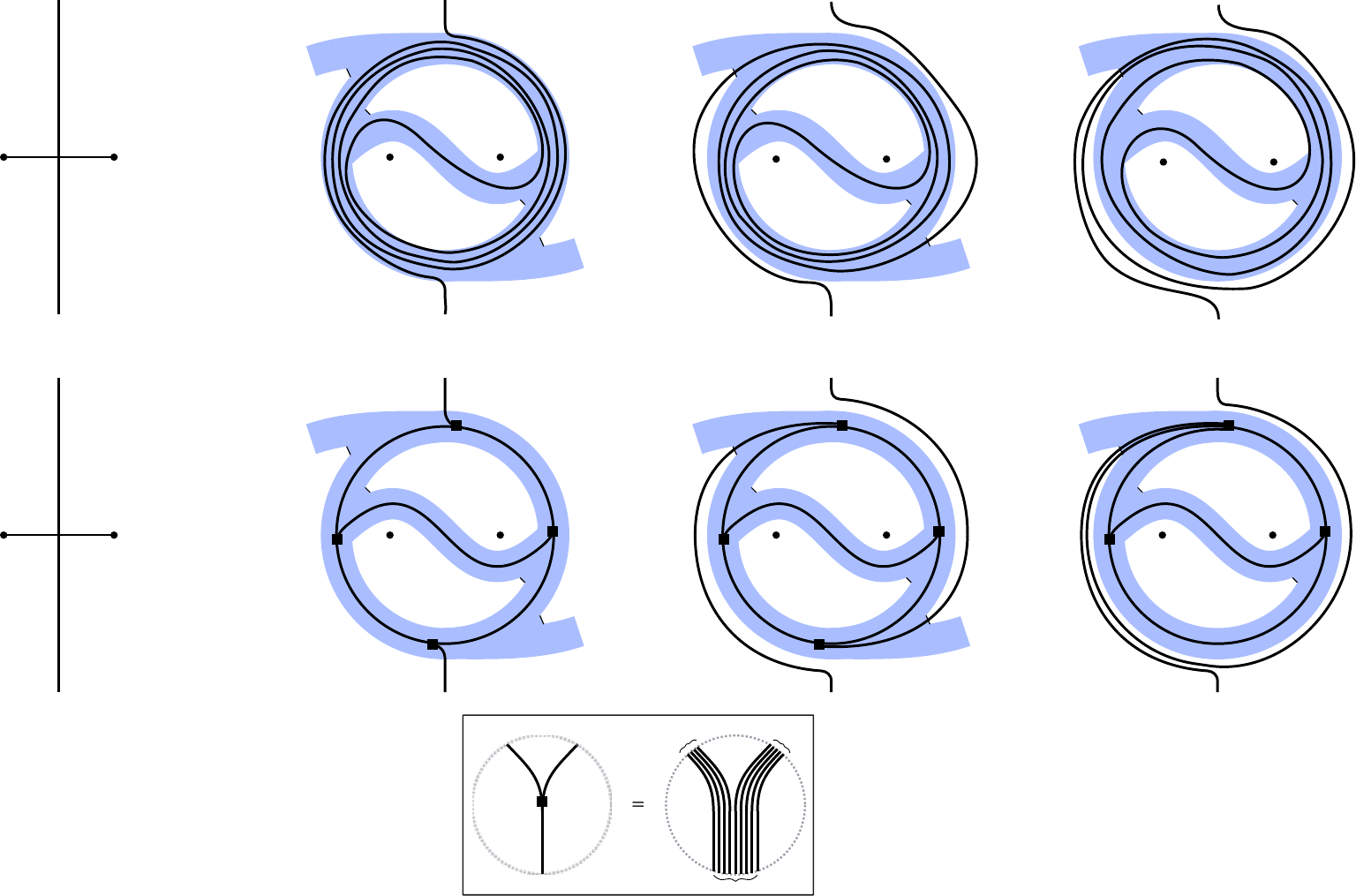}    
\caption{Illustration of Proposition \ref{prop:cover_a_tao}.
In this and following figures, an arc with a label (such as the arc coming out of the bottom of the second picture in the bottom row, marked with an \(r\)), represents not just one, but a number of parallel arcs, according to the label.
The box at the bottom of the figure illustrates that the small squares placed at various places in this and following figures represents a set of parallel arcs separating into two different sets of parallel arcs.
In particular, the square does \textit{not} represent a vertex of a graph.}
\label{fig:200615TwistingStrands}
\end{figure}

The following is Lemma 8.4 of \cite{johnson2016bridge}.

\begin{prop}\label{prop:cover_a_tao}
If \(D\) is a compressing disk above \(P_n\), then \(\pi_{n-1}\left(\partial D\right)\) covers at least one tao of \(\tau_{n-1}\), and away from those one or more taos, \(\pi_{n-1}\left(\partial D\right)\) intersects \(\tau_{n-1}\) in almost carried arcs or in fiber intervals.
\end{prop}

\begin{defn}

A subgraph \(T_i'\) of \(T_i\) is called a \textit{mini-graph} of \(T_i\) if it has the following properties.

\begin{enumerate}
\item \(T_i'\) is a union of taos, connecting arcs, and eyelets of the train graph \(T_i\).
\item Two adjacent taos of \(T_i\) are contained in \(T_i'\) if the taos' connecting arc is contained in \(T_i'\).
\item An eyelet \(E\subseteq T_i\) is contained in \(T_i'\) only if both the tao \(T\) nearest to \(E\) in \(T_i\) and every other eyelet between \(E\) and \(T\) are also contained in \(T_i'\).
\end{enumerate}
\end{defn}

\begin{defn}\label{defn:below}
Let \(T_i'\) be a mini-graph of \(T_i\). 
The mini-graph \textit{directly below} \(T_i'\) is defined to be the unique mini-graph \(T_{i-1}'\) of \(T_{i-1}\) with the following properties:
\begin{enumerate}
\item If \(T\subseteq T_i'\) is a tao or an eyelet and \(\sigma_{i-1}(T)\) intersects a tao \(T'\) of \(T_{i-1}\), then \(T'\subseteq T_{i-1}'\).
\item If \(T\subseteq T_i'\) is a tao and \(\sigma_{i-1}(T)\) intersects two taos of \(T_{i-1}\), say \(T'\) and \(T''\), then the connecting arc between \(T'\) and \(T''\) is contained in \(T_{i-1}'\).
\item If \(T\subseteq T_i'\) is a tao and \(\sigma_{i-1}(T)\) intersects an eyelet \(E'\) of \(T_{i-1}\), then \(E'\subseteq T_{i-1}'\).
\item If \(T\subseteq T_i'\) is an eyelet and  \(\sigma_{i-1}(T)\) is an eyelet of \(T_{i-1}\), then \(\sigma_{i-1}(T)\) is also an eyelet of \(T_{i-1}'\).
\end{enumerate}

Let \(i,i-j\in\{1,2,\hdots,n-1\}\), with \(i>i-j\).
Let \(T_i'\) and \(T_{i-j}'\) be mini-graphs of \(T_i\) and \(T_{i-j}\), respectively.
We say that \(T_{i-j}'\) is \textit{below} \(T_i'\) if and only if there exists a sequence of mini-graphs \(T_{i}', T_{i-1}', T_{i-2}',\hdots,  T_{i-j}'\) such that for \(k=1,2,\hdots,j\), the mini-graph \(T_{i-k-1}'\) is directly below \(T_{i-k}\).
Naturally, we will say that a mini-graph \(T_i'\) lies \textit{(directly) above} a mini-graph \(T_{i-j}'\) if and only if \(T_{i-1}'\) lies (directly) below \(T_i'\).
\end{defn}

\begin{obs}\label{obs:mini-graph_union}
If \(T_i'\) and \(T_i''\) are mini-graphs of \(T_i\), then \(T_i'\cup T_i''\) is a mini-graph of \(T_i\) as well.
\end{obs}

\begin{obs}\label{obs:mini-graph_union_below}
Suppose that \(T_i'\) and \(T_i''\) are mini-graphs of \(T_i\), that \(T_{i-1}'\) and \(T_{i-1}''\) are mini-graphs of \(T_{i-1}\), that \(T_{i-1}'\) is below \(T_i'\), and that \(T_{i-1}''\) is below \(T_i''\).
Then \(T_{i-1}'\cup T_{i-1}''\) is the mini-graph below \(T_i'\cup T_i''\).
\end{obs}

\begin{prop}\label{prop:what-below-what}
First, if \(T_{n-1}'\) is the leftmost tao of \(T_{n-1}\), and if \(T_1'\) is the mini-graph of \(T_1\) constructed by excluding from \(T_1\) only the rightmost two eyelets, then \(T_{n-1}'\) is above \(T_1'\).
Second, if \(T_{n-1}''\) is any other tao of \(T_{n-1}\) besides the leftmost tao, then \(T_{n-1}''\) is either above \(T_1\) itself or above some mini-graph \(T_1''\) of \(T_1\) constructed by excluding from \(T_1\) the leftmost eyelet and/or the rightmost eyelet.
\end{prop}
In reading through the following proof, the reader may find it helpful to use the example link in Figure \ref{fig:200320PlatExampe6-5} to help locate and visualize the various mini-graphs we discuss.
\begin{proof}
We will need to define a few more specific mini-graphs.
Define \(T_{n_b+1}'\) to be the mini-graph constructed from \(T_{n_b+1}\) by excluding from it the rightmost eyelet, tao, and connecting arc.
Define \(T_{n_b}'\) to be the mini-graph of \(T_{n_b}\) consisting of all of the taos and connecting arcs of \(T_{n_b}\) but excluding both eyelets.
Now the reader can observe that by virtue of the dimensions of the link \(L_a\), that is, since \(n_a=2m-4\), the tao \(T_{n-1}'\) is above the mini-graph \(T_{n_b+1}'\).
Further, observe that \(T_{n_b+1}'\) is directly above \(T_{n_b}'\), which is above \(T_1'\).
Therefore, \(T_{n-1}'\) is above \(T_1'\).

For the next part of the proof, suppose that \(T_{n-1}''\) is a tao of \(T_{n-1}\), but not the leftmost one.
Appealing again to the dimensions of \(L_a\), the tao \(T_{n-1}''\) is above some mini-graph \(T_{n_b+1}''\) of \(T_{n_b+1}\) which includes at least all of the taos of \(T_{n_b+1}\) but the leftmost one. 
That is, while \(T_{n_b+1}''\) \textit{may} contain the leftmost tao of \(T_{n_b+1}\),   \(T_{n_b+1}''\) \textit{does} contain all the other taos of \(T_{n_b+1}\).
It follows that \(T_{n_b+1}''\) is directly above some mini-graph \(T_{n_b}''\) of \(T_{n_b}\) which contains at least all but the leftmost tao of \(T_{n_b}\) and also the first of the two eyelets on the right side of \(T_{n_b}\).
Going down another level, \(T_{n_b}''\) must be directly above a mini-graph \(T_{n_b-1}''\) of \(T_{n_b-1}\) which contains all the taos and connecting arcs of \(T_{n_b-1}\) as well as the first and second (but not necessarily the third) eyelet on the right.

If \(n_b=2\), then  \(T_{n_b-1}=T_1\), and we can define \(T_1''=T_{n_b-1}''\), in which case the proof is finished, for the tao \(T_{n-1}''\) is above \(T_1''\), which has been shown to have the desired properties.

If \(n_b>2\), then observe that at each level from level \(n_b-1\) down to level 1, \(T_{n_b-1}''\) will be above a mini-graph consisting of all of the level's taos and connecting arcs as well as all of the level's eyelets, possibly excluding the leftmost eyelet and/or the rightmost eyelet.
Therefore \(T_{n_b-1}''\) is above some mini-graph \(T_1''\) with the desired properties, which finishes the proof.
\end{proof}

\begin{defn}
It will be helpful to name a few special types of mini-graphs.
\begin{enumerate}\setcounter{enumi}{-1}
\item If \(T\) is a tao, we will call \(T\) a \textit{Type-0} mini-graph.
\item A \textit{Type-1} mini-graph consists of a final eyelet of \(T_i\) and an adjacent tao.
\item A \textit{Type-2} mini-graph consists of a final eyelet \(E_2\), an eyelet \(E_1\) adjacent to \(E_2\), and a tao adjacent to \(E_1\).
\item A \textit{Type 3} mini-graph consists of a final eyelet \(E_3\), an eyelet \(E_2\) adjacent to \(E_3\), an eyelet \(E_1\) adjacent to \(E_2\), and a tao adjacent to \(E_1\).
\item Collectively we will refer to mini-graphs of Type 0, 1, 2, or 3 as \textit{typed} mini-graphs.
\end{enumerate}

\end{defn}

\begin{obs}\label{obs:decompose_mini-graph}
If \(T_i'\) is a mini-graph of \(T_i\), then for some positive integer \(k\),  \(T_i'\) can be decomposed into a union \(T_i'=t_1\cup t_2\cup\cdots\cup t_k\cup c\), where \(t_1,t_2,\hdots,t_k\) are typed mini-graphs and \(c\) is a (possibly empty) union of connecting arcs.
\end{obs}

\begin{obs}\label{obs:below_decomposed_mini-graph}
Suppose \(T_i'\) is a mini-graph of \(T_i\), and \(T_{i-1}'\) is the mini-graph directly below \(T_i'\).
Let \(T_i'\) be decomposed into a union \(T_i'=t_1\cup t_2\cup\cdots\cup t_k\cup c\), where \(t_1,t_2,\hdots,t_k\) are typed mini-graphs and \(c\) is a union of connecting arcs.
For each \(j\in\{1,2,\hdots,k\}\), let \(u_j\) be the mini-graph directly below \(t_j\).
Then \(T_{i-1}'=u_1\cup u_2\cup\cdots\cup u_k\).
\end{obs}

\begin{figure}[ht!]
\labellist
\tiny\hair 2pt
\pinlabel {\fontsize{4}{6}\selectfont $s-2$}  at 140 241
\pinlabel {\fontsize{4}{6}\selectfont $s-1$}  at 162 241
\pinlabel {\fontsize{4}{6}\selectfont $s-1$}  at 137 201
\pinlabel {\fontsize{4}{6}\selectfont $s-2$}  at 164 198

\pinlabel {\fontsize{4}{6}\selectfont $s-3$}  at 140 124
\pinlabel {\fontsize{4}{6}\selectfont $s-1$}  at 162 124

\pinlabel {\fontsize{4}{6}\selectfont $s-2$}  at 164 78

\pinlabel {\fontsize{4}{6}\selectfont $r-2$}  at 30 241
\pinlabel {\fontsize{4}{6}\selectfont $r-1$}  at 52 241
\pinlabel {\fontsize{4}{6}\selectfont $r-1$}  at 27 201
\pinlabel {\fontsize{4}{6}\selectfont $r-2$}  at 54 198

\endlabellist

\centering

\includegraphics[width=1\textwidth]{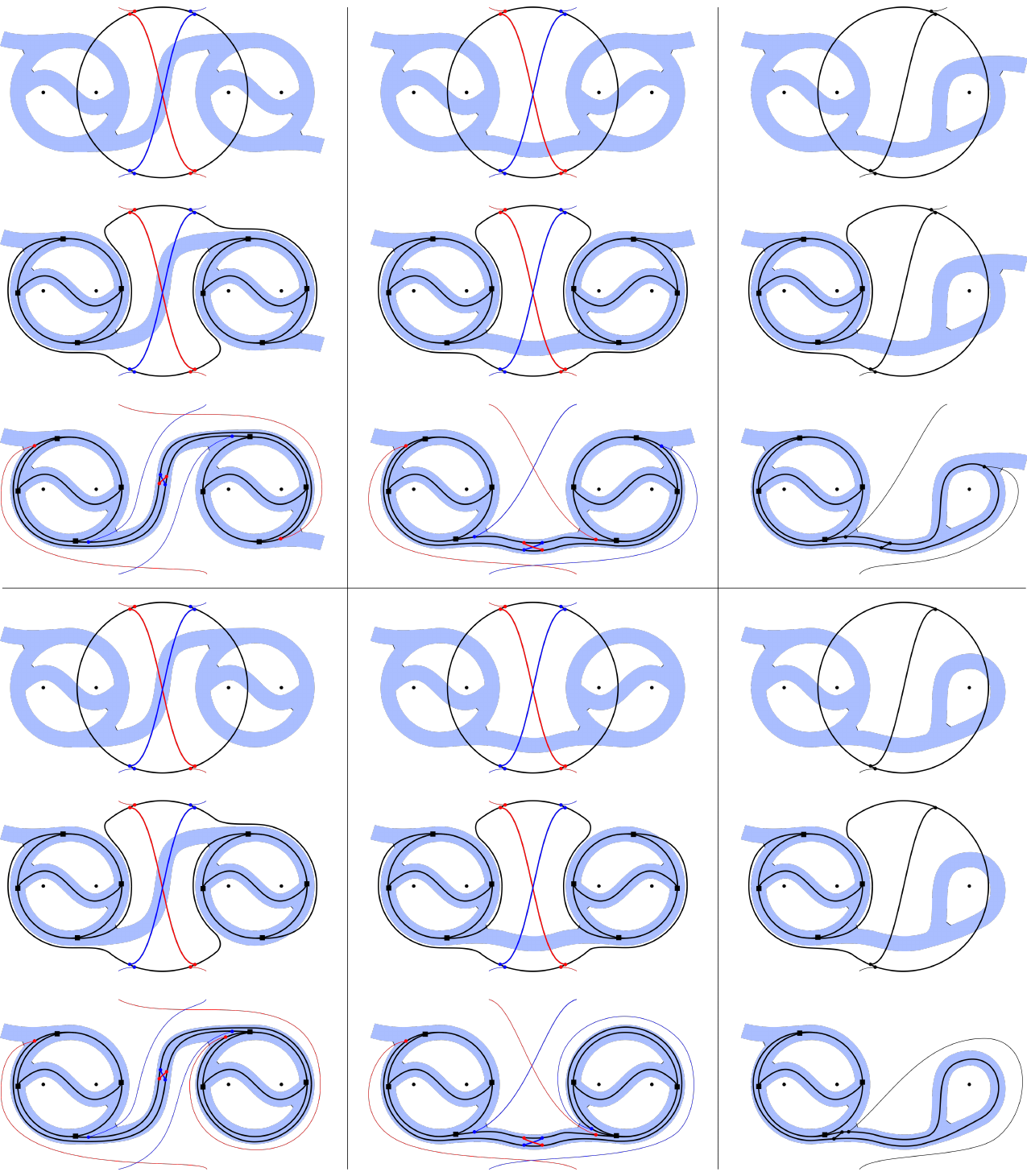}
\caption{A Type-0 mini-graph (i.e., a tao) covers the mini-graph directly below it. In each picture which includes both red and blue arcs, only either the blue arcs or the red arcs will be present, depending on the handedness of the tao.}
\label{fig:200811_Type_one_twisted_down}
\end{figure}

\begin{figure}[ht!]
\labellist
\small\hair 2pt
\pinlabel \fontsize{5}{7}\selectfont $r$ at 37 290
\pinlabel \fontsize{5}{7}\selectfont $r$ at 111 219
\pinlabel \fontsize{5}{7}\selectfont $2s$ at 222 220
\pinlabel \fontsize{5}{7}\selectfont $2(s-1)$ at 272 215
\pinlabel \fontsize{5}{7}\selectfont $2$ at 252 245
\pinlabel \fontsize{5}{7}\selectfont $r-1$ at 76 26
\pinlabel \fontsize{4}{6}\selectfont $2s-1$ [t] at 222 53
\pinlabel \fontsize{4}{6}\selectfont $2(s-1)$ [t] at 274 43
\pinlabel \fontsize{4}{6}\selectfont $2$ [t] at 260 70
\pinlabel \fontsize{5}{7}\selectfont $r-1$ at 97 297
\pinlabel \fontsize{5}{7}\selectfont $r-1$ at 47 219
\pinlabel \fontsize{5}{7}\selectfont $2(s-1)$  at 241 293
\pinlabel \fontsize{4}{6}\selectfont $2(s-1)-1$  at 243 113
\pinlabel \fontsize{4}{6}\selectfont $2s$  at 299 103
\pinlabel \fontsize{5}{7}\selectfont $2s$  at 289 293
\pinlabel \fontsize{5}{7}\selectfont $r$ at 50 118
\pinlabel \fontsize{5}{7}\selectfont $r-2$ at 93 118

\pinlabel \fontsize{5}{7}\selectfont $2(s-1)$ at 663 215
\pinlabel \fontsize{5}{7}\selectfont $2$ at 643 245
\pinlabel \fontsize{5}{7}\selectfont $2s$  at 680 293
\pinlabel \fontsize{5}{7}\selectfont $2(s-1)$  at 632 293
\pinlabel \fontsize{5}{7}\selectfont $2s$ at 613 220

\pinlabel \fontsize{4}{6}\selectfont $2(s-1)-1$  at 636 113
\pinlabel \fontsize{4}{6}\selectfont $2s$  at 690 103
\pinlabel \fontsize{4}{6}\selectfont $2s-1$ [t] at 614 53
\pinlabel \fontsize{4}{6}\selectfont $2(s-1)$ [t] at 665 43
\pinlabel \fontsize{4}{6}\selectfont $2$ [t] at 651 70

\pinlabel \fontsize{5}{7}\selectfont $r-1$ at 437 290
\pinlabel \fontsize{5}{7}\selectfont $r$ at 507 290
\pinlabel \fontsize{5}{7}\selectfont $r$ at 427 219
\pinlabel \fontsize{5}{7}\selectfont $r-1$ at 496 219

\pinlabel \fontsize{5}{7}\selectfont $r-2$ at 489 34
\pinlabel \fontsize{5}{7}\selectfont $r$ at 430 34
\pinlabel \fontsize{5}{7}\selectfont $r-1$ at 456 124
\endlabellist
\centering
\includegraphics[width=0.8\textwidth]{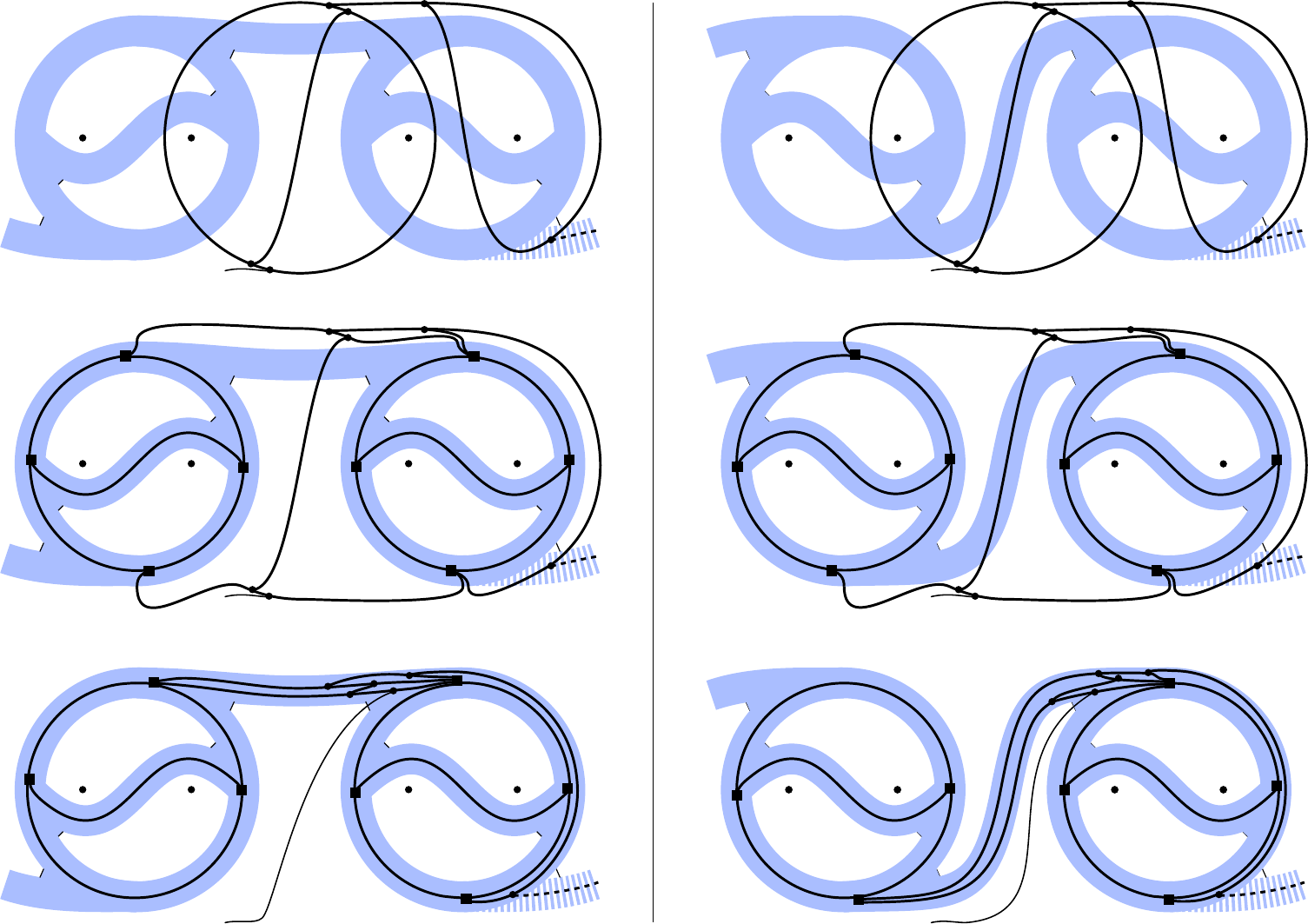}    
\caption{A Type-1 mini-graph may lie directly above two taos and their connecting arc.
If so, the Type-1 mini-graph will cover the two taos and their connecting arc.
The dashed lines of the train track diagram and of the train graph are either both present or both absent.}
\label{fig:200720_a_tao_and_one_eyelet_over_two_taos}
\end{figure}

\begin{figure}[ht!]
\labellist
\small\hair 2pt
\pinlabel \fontsize{4}{6}\selectfont $r-1$ at 47 294
\pinlabel \fontsize{4}{6}\selectfont $r-1$  at 90 120
\pinlabel \fontsize{4}{6}\selectfont $r$  at 45 30
\pinlabel \fontsize{4}{6}\selectfont $r-2$  at 97 30
\pinlabel \fontsize{4}{6}\selectfont $r$  at 125 272
\pinlabel \fontsize{4}{6}\selectfont $r$ at 43 212
\pinlabel \fontsize{4}{6}\selectfont $r-1$ at 100 212
\endlabellist

\centering
\includegraphics[width=.4\textwidth]{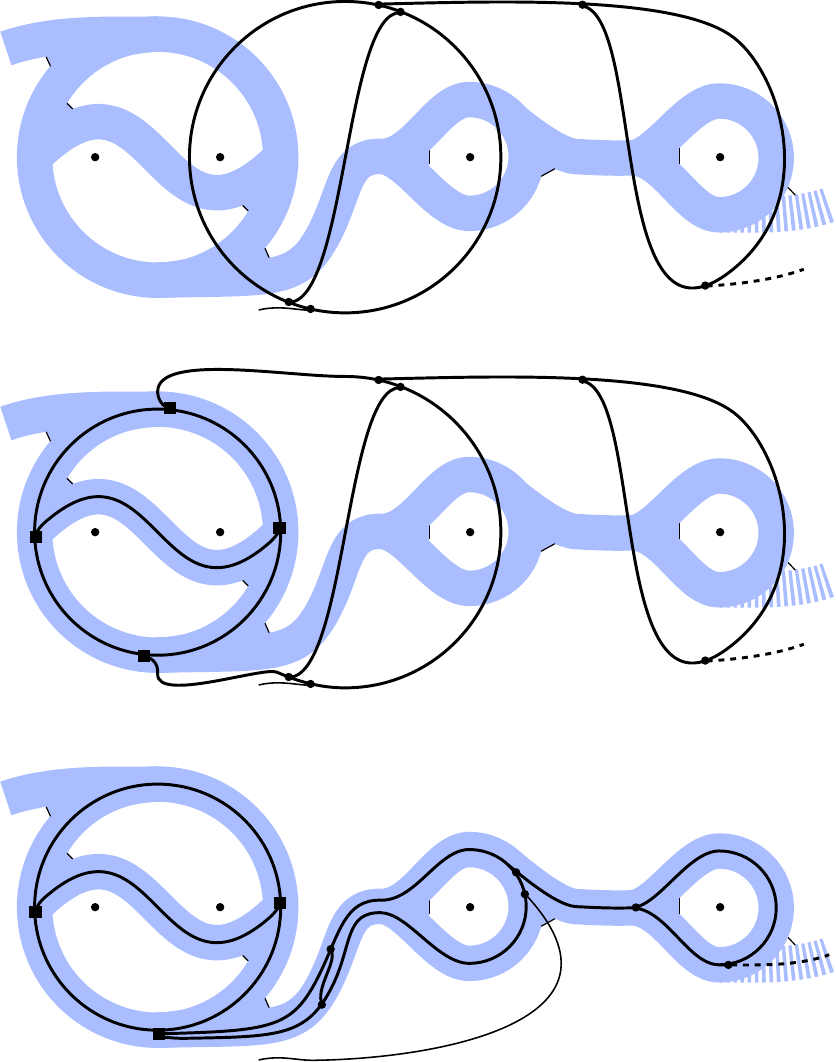}
\caption{A Type-1 mini-graph may lie directly above a Type-2 mini-graph, in which case the former will cover the latter.}
\label{fig:200720_type_one_over_type_two.pdf}
\end{figure}

\begin{figure}[ht!]
\labellist
\small\hair 2pt
\pinlabel {\fontsize{5}{7}\selectfont $r$} at 37 280
\pinlabel {\fontsize{5}{7}\selectfont $r-1$} at 100 282
\pinlabel {\fontsize{5}{7}\selectfont $r-1$} at 54 207
\pinlabel {\fontsize{5}{7}\selectfont $r-1$} at 76 21
\pinlabel {\fontsize{5}{7}\selectfont $r$} at 37 99
\pinlabel {\fontsize{5}{7}\selectfont $r$} at 107 208
\pinlabel {\fontsize{5}{7}\selectfont $r-2$} at 103 103
\endlabellist
\centering
\includegraphics[width=.5\textwidth]{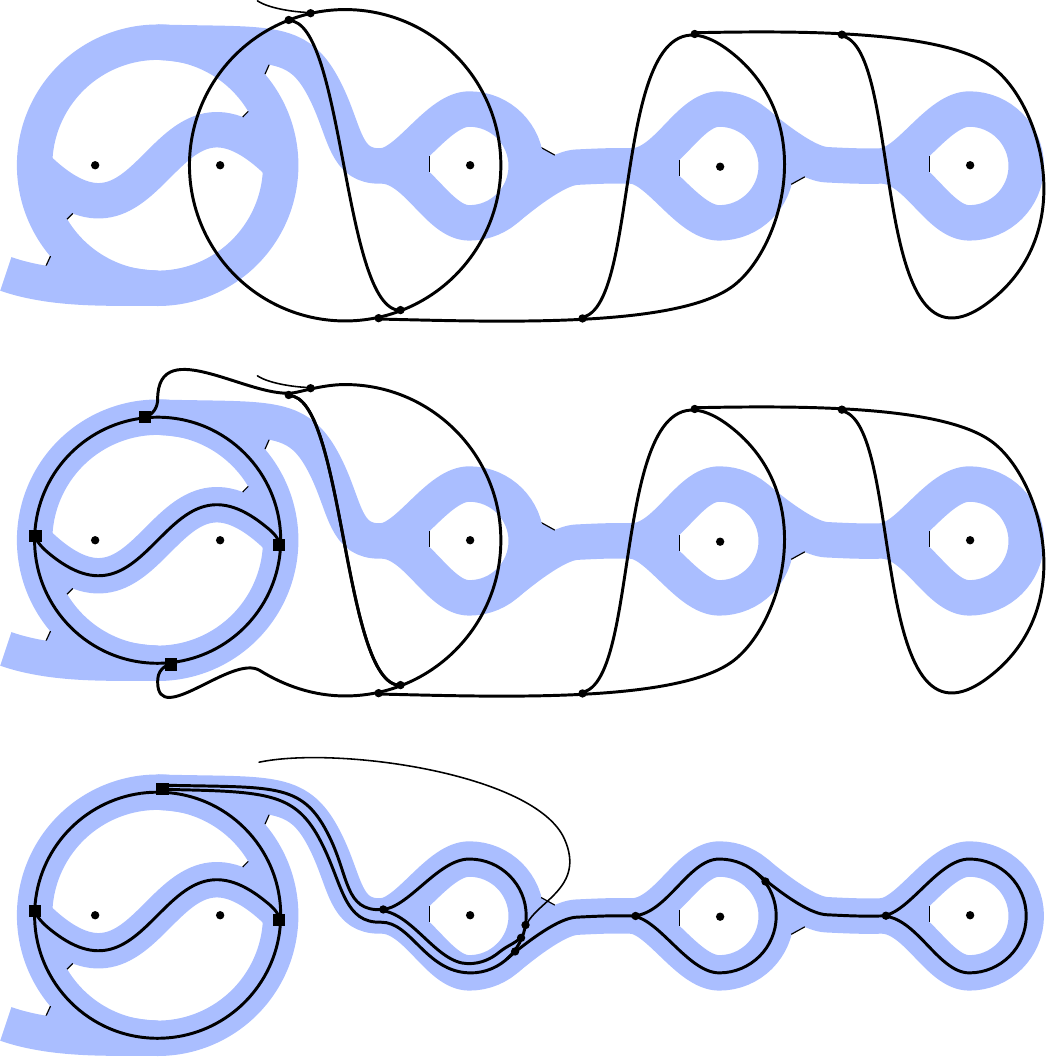}    
\caption{
A Type-2 mini-graph always lies directly above and covers a Type-3 mini-graph.
}
\label{fig:200803_1_tao_and_2_eyelets_over_one_tao_and_3_eyelets}
\end{figure}

\begin{figure}[ht!]
\labellist
\small\hair 2pt
\pinlabel {\fontsize{5}{7}\selectfont $r-1$} at 52 286
\pinlabel {\fontsize{5}{7}\selectfont $r$} at 121 272
\pinlabel {\fontsize{5}{7}\selectfont $r-1$} at 102 213
\pinlabel {\fontsize{5}{7}\selectfont $r$} at 42 213
\pinlabel {\fontsize{4}{6}\selectfont $2(s-1)$} at 240 210

\pinlabel
{\fontsize{4}{6}\selectfont $2s-1$} at 232 100
\pinlabel{\begin{rotate}{-30}
{\fontsize{4}{6}\selectfont $2s-2$}\end{rotate}} at 262 115
\pinlabel{\begin{rotate}{-30}
{\fontsize{4}{6}\selectfont $2s-2$}\end{rotate}} at 879 115
\pinlabel 
{\fontsize{4}{6}\selectfont $2s-3$} at 237 30
\pinlabel 
{\fontsize{4}{6}\selectfont $2s$} at 286 30
\pinlabel {\fontsize{4}{6}\selectfont $2s$} at 299 230
\pinlabel {\fontsize{5}{7}\selectfont $r-2$} at 94 29
\pinlabel {\fontsize{5}{7}\selectfont $r$} at 36 36
\pinlabel {\fontsize{5}{7}\selectfont $2$} at 250 77
\pinlabel {\fontsize{5}{7}\selectfont $2$} at 253 259

\pinlabel
{\fontsize{5}{7}\selectfont $2$} at 867 77
\pinlabel {\fontsize{5}{7}\selectfont $2$} at 870 259

\pinlabel {\fontsize{4}{6}\selectfont $2s$} [bl] at 210 273
\pinlabel {\fontsize{4}{6}\selectfont $2(s-1)$} [bl] at 244 273
\pinlabel {\fontsize{5}{7}\selectfont $r-1$} at 85 114

\pinlabel {\fontsize{5}{7}\selectfont $r$} at 659 286
\pinlabel {\fontsize{5}{7}\selectfont $r-1$} at 709 286
\pinlabel {\fontsize{5}{7}\selectfont $r-1$} at 668 212
\pinlabel {\fontsize{5}{7}\selectfont $r$} at 719 205

\pinlabel {\fontsize{5}{7}\selectfont $r$} at 659 106
\pinlabel {\fontsize{5}{7}\selectfont $r-1$} at 700 20
\pinlabel {\fontsize{5}{7}\selectfont $r-2$} at 710 106

\pinlabel{\fontsize{4}{6}\selectfont $2(s-1)$} at 857 210
\pinlabel {\fontsize{4}{6}\selectfont $2s$} at 916 230
\pinlabel {\fontsize{4}{6}\selectfont $2s$} [bl] at 827 273
\pinlabel {\fontsize{4}{6}\selectfont $2(s-1)$} [bl] at 861 273

\pinlabel{\fontsize{4}{6}\selectfont $2s-1$} at 850 105
\pinlabel 
{\fontsize{4}{6}\selectfont $2s-3$} at 854 30
\pinlabel 
{\fontsize{4}{6}\selectfont $2s$} at 903 30
\endlabellist

\centering
\includegraphics[width=1\textwidth]{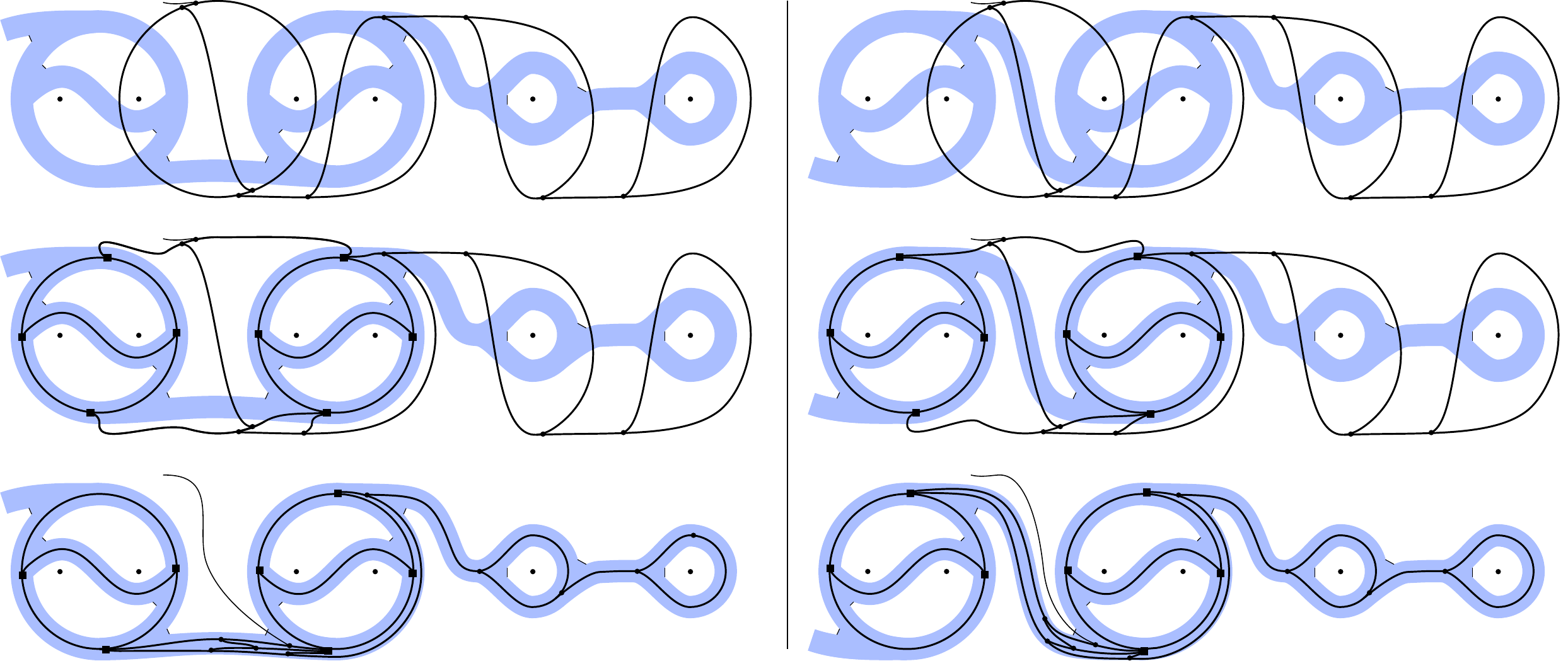}
\caption{A Type-3 mini-graph always lies directly above and covers a Type-2 mini-graph (assuming there exists a level below the level of the Type-3 minigraph).}
\label{fig:200803_1_tao_and_3_eyelets_over_2_taos_and_2_eyelets}
\end{figure}

\begin{prop}\label{prop:mini-graph_covers_the_one_below}
For each \(i=2,3,\hdots,n-1\), if \(T_i'\) is a mini-graph of \(T_i\), then \(\pi_{i-1}(T_i')\) covers the mini-graph directly below \(T_i'\). 
\end{prop}

\begin{proof}
Fix \(2\leq i\leq n-1\).
Let \(T_i'\subseteq T_i\) be a mini-graph, and let \(T_{i-1}'\) be the mini-graph directly below \(T_i'\).
We will prove this proposition by proving several special cases which will lead us to the general result.

To begin, consider the case in which \(T_i'\) is a Type-0 mini-graph (a tao).
The mini-graph \(T_{i-1}'\) may consist of two adjacent taos and their connecting arc, as in the first or second column of pictures in Figure \ref{fig:200811_Type_one_twisted_down}, or \(T_{i-1}'\) may instead be a Type-1 mini-graph, as depicted in the third column of pictures in Figure \ref{fig:200811_Type_one_twisted_down}.
In any case, the result of the isotopy of the bridge sphere from level \(i\) to level \(i-1\) is shown from the top row of pictures to the second row, or from the fourth row of pictures to the fifth row.

Observe that an isotopy of \(\pi_{i-1}(T_i')\) in \(P_{i-1}\) (the result of which is shown in the third and sixth rows of pictures in Figure \ref{fig:200811_Type_one_twisted_down}) shows how we may push \(\pi_{i-1}(T_i')\) into \(\tau_{i-1}\) so that \(\pi_{i-1}(T_i')\) covers \(T_{i-1}'\).

Next, if \(T_i'\) is a Type-1 mini-graph, then either \(T_{i-1}'\) is a pair of taos  (pictured in Figure \ref{fig:200720_a_tao_and_one_eyelet_over_two_taos}) or \(T_{i-1}'\) is a Type-2 mini-graph (pictured in Figure \ref{fig:200720_type_one_over_type_two.pdf}).
Either way, the figures illustrate that \(\pi_{i-1}(T_i')\) covers \(T_{i-1}'\).

Now suppose \(T_i'\) is a Type-2 mini-graph.
In this case, \(T_{i-1}'\) must be a Type-3 mini-graph.
Figure \ref{fig:200803_1_tao_and_2_eyelets_over_one_tao_and_3_eyelets} depicts this case and shows that \(\pi_{i-1}(T_i')\) covers \(T_{i-1}'\).

Finally, suppose \(T_i'\) is a Type-3 mini-graph.
It follows that \(T_{i-1}'\) is a Type-2 mini-graph, as depicted in Figure \ref{fig:200803_1_tao_and_3_eyelets_over_2_taos_and_2_eyelets}, which shows that as before, \(\pi_{i-1}(T_i')\) covers \(T_{i-1}'\).

\begin{obs}\label{obs:connecting_arc_end_almost_carried}
Notice that in each of the cases above, if \(c\) is a connecting arc of \(T_i\) which is attached to \(T_i'\) at vertex \(v\), then the connected component of \(c\cap\tau_{i-1}\) which contains \(v\) is almost carried by \(\tau_{i-1}\).
\end{obs}

Now that we have proven the proposition for cases in which \(T_i'\) is a typed mini-graph, we are ready to prove it in the general case where \(T_i'\) is an arbitrary mini-graph.
According to Observation \ref{obs:decompose_mini-graph}, we can view \(T_i'\) as a union \(T_i'=t_1\cup t_2\cup\cdots\cup t_k\cup c\) of typed mini-graphs and connecting arcs.
For each \(j\in\{1,2,\hdots,k\}\), define \(u_j\) to be the mini-graph of \(T_{i-1}\) below \(t_j\).
By Observation \ref{obs:below_decomposed_mini-graph}, \(T_{i-1}'=u_1\cup u_2\cup\cdots\cup u_k\).
The special cases above demonstrate that for each \(j\in\{1,2,\hdots,k\}\), the mini-graph \(u_j\) is covered by \(t_j\), so it follows that \(u_1\cup u_2\cup\cdots\cup u_k\) is covered by \(\pi_{i-1}(t_1\cup t_2\cup\cdots\cup t_k)\).
Further, if \(c_0\) is one of the connecting arcs of \(c\), then by Observation \ref{obs:connecting_arc_end_almost_carried},
\(c_0\) is also almost carried by \(\tau_{i-1}\).
Therefore, since \(u_1\cup u_2\cup\cdots\cup u_k\) is covered by \(\pi_{i-1}(t_1\cup t_2\cup\cdots\cup t_k)\) and \(c\) is almost carried by \(\tau_{i-1}\), we can conclude that \(u_1\cup u_2\cup\cdots\cup u_k\) is covered by \(\pi_{i-1}(t_1\cup t_2\cup\cdots\cup t_k\cup c)\), or more simply, \(T_{i-1}'\) is covered by \(\pi_{i-1}(T_i')\).
\end{proof}

\begin{cor}
Let \(\ell\) be a loop which covers a mini-graph \(T_i'\subseteq T_i\), and let \(T_{i-1}'\) be the mini-graph directly below \(T_i'\).
The loop \(\pi_{i-1}(\ell)\) covers \(T_{i-1}'\).
\end{cor}

\begin{proof}
Let \(J\) be an interval fiber of \(\tau_{i-1}'\) which \(T_{i-1}'\) intersects.
By Proposition \ref{prop:mini-graph_covers_the_one_below}, \(\pi_{i-1}(T_i')\) covers \(T_{i-1}'\), and so by the definition of covering, \(J\) is also intersected by \(\pi_{i-1}(T_{i}')\).

Let \(p\) be a point of \(\left(\pi_{i-1}(T_i')\right)\cap J\), and let \(I\) be the interval fiber of \(\pi_{i-1}(\tau_i)\) which contains \(p\).
By Proposition \ref{prop:train-track-in-train-track}, \(I\subseteq J\).
Further, since \(\ell\) covers \(T_i'\), \(\ell\) must by definition intersect \(I\).
It follows that since \(I\subseteq J\), \(\ell\) intersects \(J\).
\end{proof}

\begin{cor}\label{cor:loop-cover-below}
Let \(i_1<i_2\), and let \(T_{i_1}'\subseteq T_{i_1}\) be the mini-graph below a mini-graph \(T_{i_2}'\subseteq T_{i_2}\).
If \(\ell\) is a loop which covers \(T_{i_2}'\subseteq T_{i_2}\), then \(\pi_{i_1}(\ell)\) covers \(T_{i_1}'\).
\end{cor}

Recall the notation of Proposition \ref{prop:what-below-what}.
The leftmost blue disk $B$ above \(P_n\) is the only disk whose boundary loop covers \(T_{n-1}'\) but no other taos. 
In contrast, the boundary of every red disk above \(P_n\) must cover at least one of the other taos.
The next corollary then follows from Proposition \ref{prop:what-below-what} and Corollary \ref{cor:loop-cover-below}.

\begin{cor}\label{cor:what-B-and-reds-cover}
The boundary of the blue disk \(B\) above \(P_n\) covers \(T_1'\) (the mini-graph defined in Proposition \ref{prop:what-below-what}), and the boundary of every red disk above \(P\) covers either \(T_1\) or some mini-graph \(T_1''\) constructed from \(T_1\) by excluding the leftmost and/or the rightmost eyelet of \(T_1\).

\end{cor}

An almost carried loop $\ell$ that covers enough taos and eyelets is very beneficial in the sense that its presence allows us to predict the behavior of loops which are disjoint from $\ell$.

\begin{rem}
The following is Lemma 6.5 from Johnson and Moriah.
\end{rem}

\begin{lem}\label{lem:ACandDisjoint}
If $\ell$ is a loop in \(P_i\) that covers a mini-graph \(T_i''\) of \(T_i\), and if $\ell'$ is another loop in \(P_i\) disjoint from $\ell,$ then $\ell'$ can be isotoped to be almost carried by $\tau_i''$, the train track diagram corresponding to \(T_i''\).
\end{lem}

\begin{figure}
\centering
\includegraphics[width=.7\textwidth]{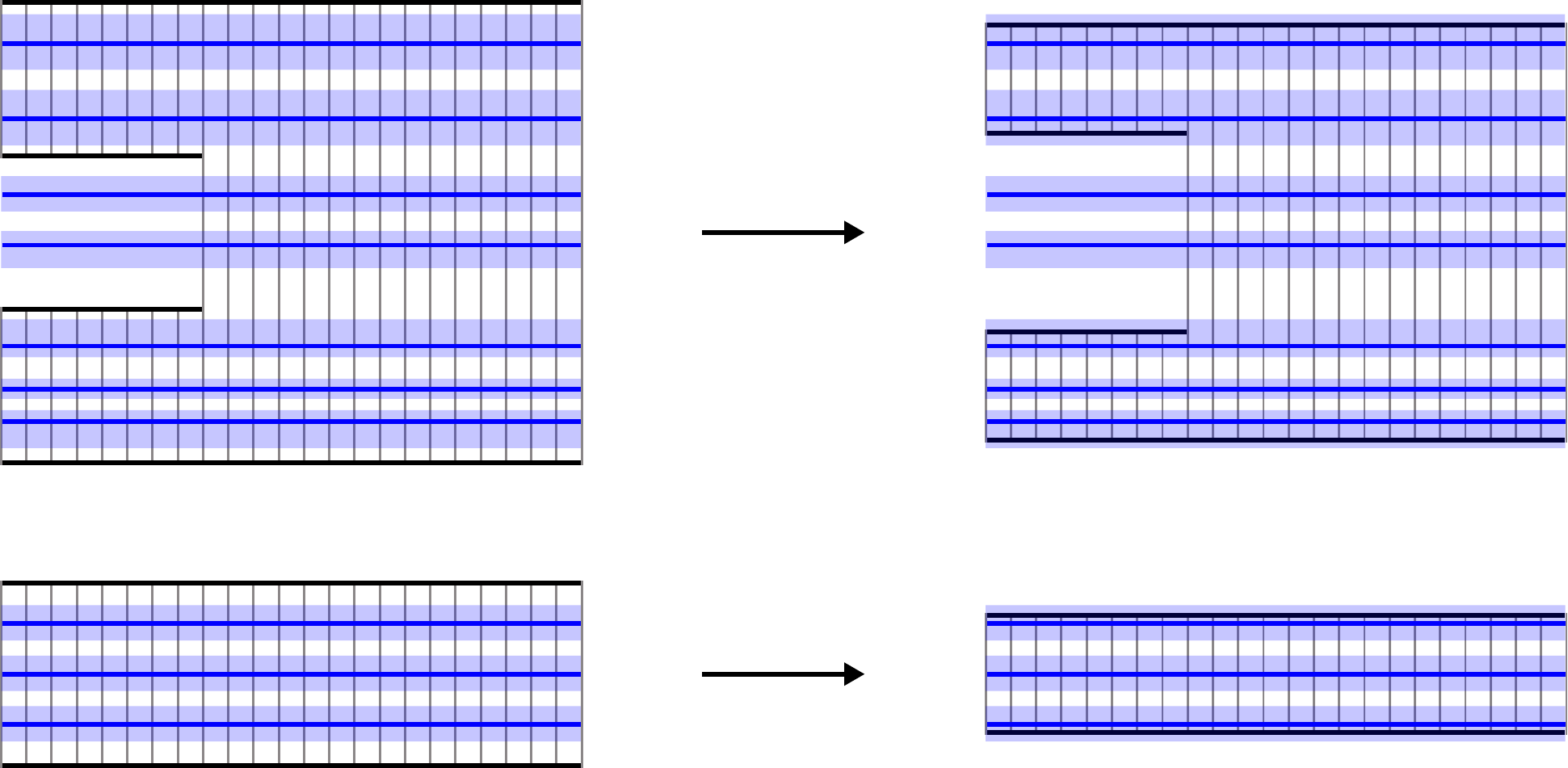}
\caption{The loop \(\ell\), shown in blue, covers the train track \(\tau_i''\).
We perform a small isotopy of \(\tau_i''\) which takes each arc of \(\partial\tau_i''\) consisting of fiber endpoints into a regular neighborhood \(N(\ell)\) of \(\ell\).
}
\label{fig:200501_tau_isotopy}
\end{figure}

\begin{proof}
Let \(N(\ell)\) be an open regular neighborhood of \(\ell\) disjoint from \(\ell'\).
Since \(\tau_i''\) is covered by \(\ell\), every interval fiber of \(\tau_i''\) intersects \(N(\ell)\).

We perform a small isotopy of \(\tau_i''\) with the following properties:
\begin{enumerate}
    \item The image of each interval fiber of \(\tau_i''\) at each moment of the isotopy is a subinterval of the original interval fiber.
    \item The endpoints of each interval fiber of \(\tau_i''\) never intersect \(\ell\) throughout the isotopy.
    \item After the isotopy, both endpoints of every interval fiber of \(\tau_i''\) lie in \(N(\ell)\).
\end{enumerate}
The result of this isotopy is illustrated in Figure \ref{fig:200501_tau_isotopy}.
The point is that each arc of \(\partial\tau_i''\) consisting of endpoints of interval fibers gets pushed into \(N(\ell)\).
Now each component of \(\tau_i''\backslash N(\ell)\) is a band in \(P_i\) fibered by intervals (each of which is a subinterval of the original interval fibers of \(\tau_i''\)). The two interval fibers contained in the boundary of a band will be referred to as \textit{exits}.
Note that topologically, each of these bands is a closed disk.

Now we isotope \(\ell'\) in \(P_i\backslash N(\ell)\) to intersect these bands minimally.
Suppose \(\ell'\cap\tau_i''=\varnothing\).
Then by Remark \ref{rem:disjoint}, \(\ell'\) is almost carried by \(\tau_i''\).
If \(\ell'\cap\left(P_i\backslash\interior{\tau_i}''\right)\) is empty (that is, \(\ell'\) lies completely in a band), then we can perform an isotopy of \(\ell'\), pushing it out of the band through an exit, contradicting minimality.

Assume then that both \(\ell'\cap\tau_i''\) and \(\ell'\cap\left(P_i\backslash\interior{\tau_i}''\right)\) are nonempty.
Consider a component \(\alpha\) of \(\ell'\cap\left(\tau_i''\backslash N(\ell)\right)\).
The component \(\alpha\) must be an arc properly embedded in a band.
Since \(\ell'\cap N(\ell)\) is empty, the endpoints of \(\alpha\) must lie in exits.
Further, the endpoints of \(\alpha\) must lie on \textit{different} exits, for otherwise \(\alpha\) could be isotoped out of the band, reducing the number of components of \(\ell'\cap\left(\tau_i''\backslash N(\ell)\right)\), again contradicting minimality.
Thus \(\alpha\) is an arc that travels through a band from one exit to another, so it follows that the arc \(\alpha\) can then be made transverse to each fiber of the band, and thus \(\ell'\) fulfills conditions \ref{def:almost_carried_condition_1} and \ref{def:almost_carried_condition_2} of the definition of almost carried.
(Note that \(\alpha\) also vacuously fulfills  condition \ref{def:almost_carried_condition_3}.)

Now let \(\beta\) be a component of \(\ell'\cap \left(P_i\backslash\interior{\tau_i}''\right)\).
The endpoints of \(\beta\) lie on \(\partial\tau_i''\).
Since the interval fiber endpoints of \(\tau_i''\) all lie in \(N(\ell)\) and \(\ell'\cap N(\ell)\) is empty, the endpoints of \(\beta\) must lie in exits.
Clearly then, \(\beta\) satisfies conditions \ref{def:almost_carried_condition_1} and \ref{def:almost_carried_condition_2} of the definition of almost carried.
The arc \(\beta\) cannot be parallel in \(P_i\backslash N(\ell)\) to a subarc of an exit, for the parallelism would guide an isotopy of \(\ell'\) through a band of \(\tau_i''\backslash N(\ell)\), thereby removing two components of intersection between \(\ell'\) and the bands, once again contradicting minimality.
Thus \(\ell'\) fulfills condition \ref{def:almost_carried_condition_3} of the definition of almost carried.

We have shown that each arc of \(\ell\cap\tau_i''\) and each arc of \(\ell\cap\left(P_i\backslash\interior{\tau_i}''\right)\) satisfies conditions \ref{def:almost_carried_condition_1}, \ref{def:almost_carried_condition_2} and \ref{def:almost_carried_condition_3} of Definition \ref{defn:almostcarried_arc}.
Remark \ref{rem:fulfill_conditions_1-2} tells us that each such arc can be isotoped to be almost carried by \(\tau_i''\).
Therefore \(\ell'\) is by definition almost carried by \(\tau_i''\).
\end{proof}


Henceforth, on $P_n$, we label the punctures as $p_1,p_2,\hdots,p_{2m}$ in order from left to right. We label the straight arcs connecting the puncture labelled $p_{2k-1}$ to the puncture labelled $p_{2k}$ as $\beta^k$. Finally, we label the straight arcs connecting the puncture labelled $p_{2k}$ to the puncture labelled $p_{2k+1}$ as $\gamma^k$.

\begin{lem}\label{lem:not_almost_carried}
As above, let \(T_1'\) be the mini-graph of \(T_1\) constructed by excluding the rightmost two eyelets, and let \(\tau_1'\) be the train track diagram corresponding to \(T_1'\).
Let \(T_1''\) be a mini-graph of \(T_1\) constructed by possibly excluding the leftmost and/or rightmost eyelet of \(T_1\), and let \(\tau_1''\) be the train track diagram corresponding to \(T_1''\).
Neither \(\tau_1'\) nor \(\tau_1''\) almost carries the boundary of any red cap for the rightmost bridge arc \(\alpha_-^m\).
\end{lem}


\begin{proof}

We prove this by contradiction.
Let \(R\) be a red cap for \(\alpha_-^m\) (which implies \(R\) is not isotopic to the blue cap \(B'=dD_-^m\)), and assume \(\partial R\) is almost carried by \(\tau_1'\) or \(\tau_1''\).
Assume \(R\) is in minimal position with respect to the vertical bridge disks below \(P_1\).

\begin{figure}
\centering
\labellist \tiny\hair 2pt
\pinlabel {
\begin{rotate}{55}
{\tiny
Case 1
}
\end{rotate}
} at 468 34
\pinlabel {
\begin{rotate}{40}
{\tiny
Case 2
}
\end{rotate}
} at 485 -5
\pinlabel {
\begin{rotate}{45}
{\tiny
Case 3
}
\end{rotate}
} at 522 20
\pinlabel {\fontsize{5}{6}\selectfont Case 4}
at 570 22
\pinlabel {\fontsize{5}{6}\selectfont Case 5}
at 589 79
\endlabellist
\includegraphics[width=.95\textwidth]{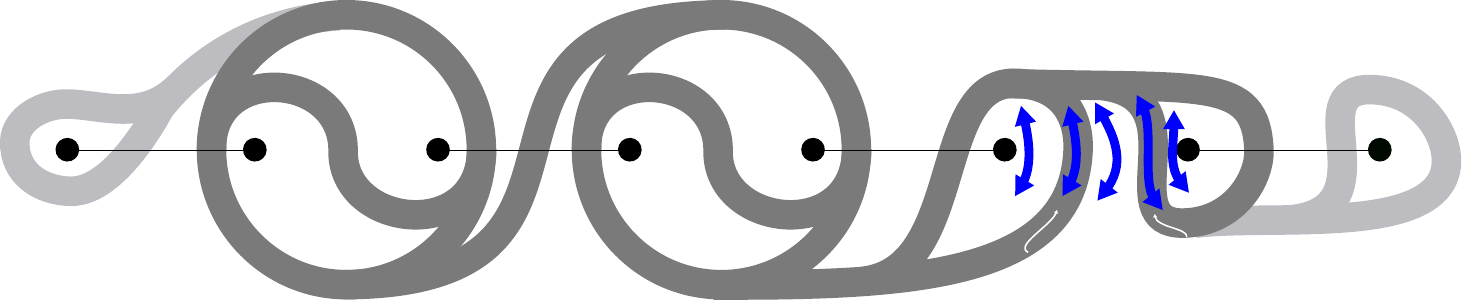}
\caption{There are only five ways for a loop or arc to pass between the rightmost two bridge arcs while remaining almost carried by \(\tau_1''\).
The leftmost and rightmost eyelets in this figure are shaded a lighter color to remind the reader that they may or may not be present in \(\tau_1''\).}
\label{fig:211207NotAlmostCarriedProof}
\end{figure}

We will first establish that \(R\) must intersect a lower vertical bridge disk.
Let \(\Gamma_1\) be the unique straight line segment in \(P_1\) which both contains all of the punctures and has \(p_1\) and \(p_{2m}\) as endpoints.
If \(\partial R\) does not intersect \(\Gamma_1\), then \(R\) is trivial, a contradiction.
So \(R\) must intersect \(\Gamma_1\).
Suppose \(\partial R\) passes between the $i$-th and the $(i+1)$-th vertical bridge disks for $i\in \lbrace 1,2,\hdots,m-3,m-2\rbrace$.
(This is equivalent to supposing that \(\partial R\) intersects \(\sigma_1(\gamma_i)\).)
Since \(\sigma_1(\gamma_i)\) is surrounded by a tao, and since \(\partial R\) is almost carried by the train track \(\tau_1'\) or \(\tau_1''\), the loop \(\partial R\) is 
forced to intersect one of the lower vertical bridge disks.
Suppose \(\partial R\) passes between the $(m-1)$-th and $m$-th vertical bridge disk.
If \(\partial R\) does not also pass between some other pair of bridge disks, then \(R\) is the blue disk \(B'\), a contradiction.
Then \(\partial R\) must pass between two of the other bridge disks, and so according to the argument above, \(\partial R\) will intersect one of the lower bridge disks.
Thus we have established that \(R\) is not disjoint from the vertical bridge disks.

Let \(\Lambda=R\cap\left(\cup_{i=1}^m D_-^i\right)\), which is nonempty, as shown above.
Since \(R\) is in minimal position with respect to the bridge disks, \(\Lambda\) contains no loops of intersection, and so \(\Lambda\) is a collection of arcs.
Let \(\gamma\subseteq\Lambda\) be an outermost arc in \(R\), cutting off an outermost disk \(R_{\text{out}}\) from \(R\).
Define \(q=R_{\text{out}}\cap P_1\).
Then \(q\) is an arc in \(P_1\) with endpoints on \(\beta_i\) for some \(i\).
Let \(\beta_*\) be the subarc of \(\beta_i\) which shares its endpoints with \(q\).

We examine what \(q\) can look like (and eventually arrive at a contradiction).
First if \(q\) never crosses the arc \(\Gamma_1\), then \(q\) would define an isotopy of \(R\) through which we could decrease the number of components of \(\Lambda\), contradicting the fact that \(R\) is in minimal position with respect to the lower vertical bridge disks.
Since the interior of \(q\) is by definition disjoint from the lower vertical bridge disks, \(q\) must therefore intersect $\sigma_1(\gamma^k)$ for some $k$. 
If \(q\) passes between the $j$-th and the $(j+1)$-th vertical bridge disks for $j \in \lbrace 1,2,\hdots,m-3,m-2\rbrace$, then as above, since \(q\) is almost carried and since the point \(q\cap\sigma_1(\gamma_k)\) is surrounded by a tao, the train track \(\tau_1'\) or \(\tau_1''\) (whichever is relevant) will force \(q\) to intersect two distinct vertical bridge disks, a contradiction.
Therefore \(q\) must pass between the $(m-1)$-th and the $m$-th bridge disks.


Suppose that \(\partial R\) is almost carried by \(\tau_1''\).
There are five ways that \(q\), as an almost carried arc, can pass between the $(m-1)$-th and the $m$-th bridge disks, and they are illustrated in Figure \ref{fig:211207NotAlmostCarriedProof}.
In Cases 1 and 2, since \(q\) is almost carried, the endpoints of \(q\) must lie on both the $(m-1)$-th and the $(m-2)$-th bridge disks, but that contradicts the definition of \(q\), for both endpoints of \(q\) must lie on the same vertical bridge disk.
Similarly, in Cases 4 and 5, the endpoints of \(q\) must lie on both the $(m-1)$-th and the $m$-th bridge disks, which contradicts the definition of \(q\) in the same way. 
Therefore the only case remaining is Case 3, in which we see \(q\) must enter a switch of \(\tau_1''\) and go on to intersect the $(m-1)$-th bridge disk.
By the definition of \(q\), the other endpoint of \(q\) must also intersect the $(m-1)$-th bridge disk on the same side.
Thus \(q\) has these properties:
\begin{itemize}
    \item The arc \(q\) has both endpoints on the $(m-1)$-th bridge disk.
    \item Of the two bands of \(\tau_1''\) going through the $(m-1)$-th bridge disk, an endpoint of \(q\) is contained in the rightmost one. 
    \item The arc $q$ is almost carried by \(\tau_1''\).
    \item The interior of $q$ does not intersect any bridge disks below \(P_1\).
    \item The arc $q$ leaves the $(m-1)$-th bridge disk in the same direction from both endpoints.
\end{itemize}
Up to isotopy, there is only one arc that has these properties, and it is depicted in Figure \ref{fig:220907_not-almost-carried-Q}.
Let \(a\) and \(b\) be the left and right endpoints, respectively, of \(q\).

\begin{figure}
\centering
\includegraphics[width=.6\textwidth]{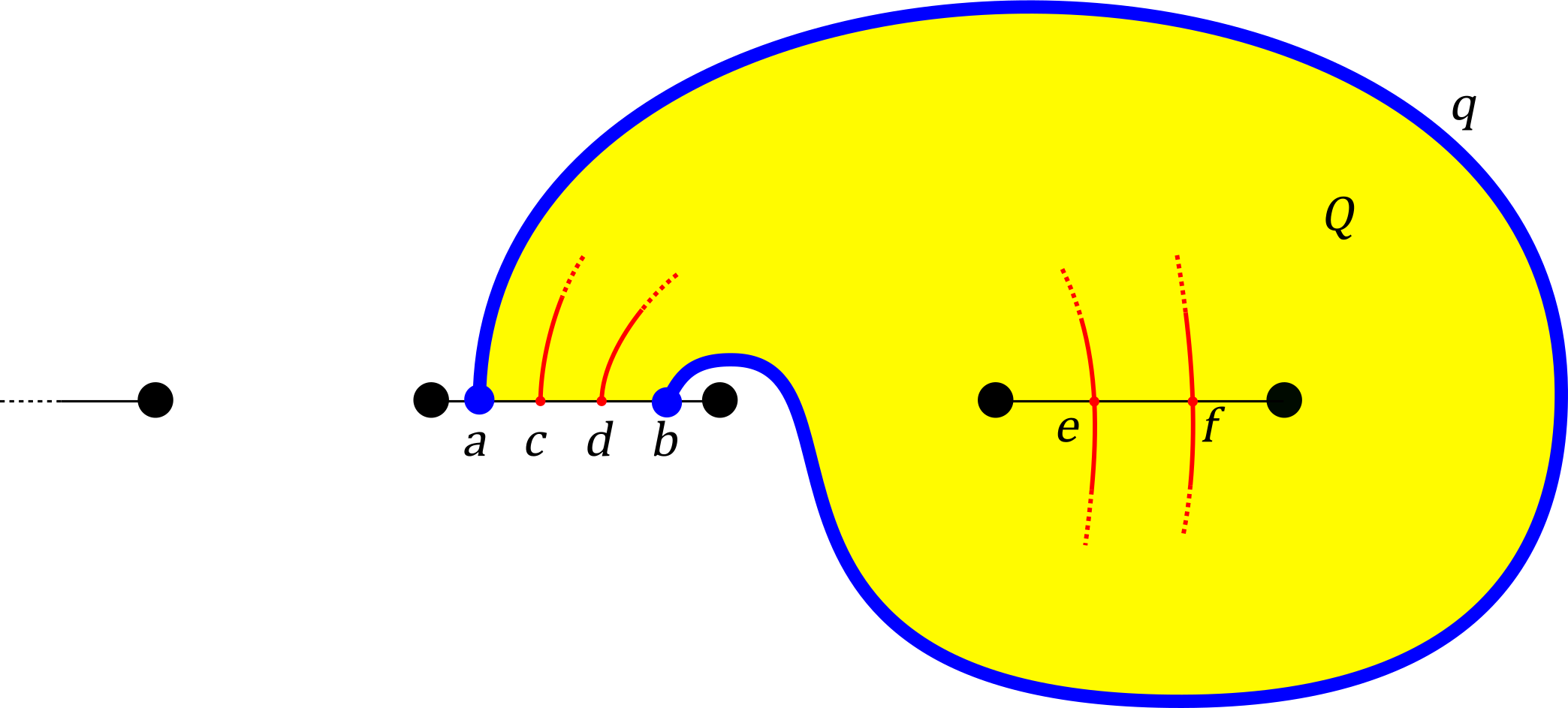}
\caption{The arc \(q\) and the disk \(Q\).}
\label{fig:220907_not-almost-carried-Q}
\end{figure}

The loop \(q\cup\beta_*\) cuts \(P_1\) into two punctured disks, one of which contains exactly two punctures: the endpoints of \(\beta_m\).
Call this 2-punctured disk \(Q\).
(See Figure \ref{fig:220907_not-almost-carried-Q}.)
Observe that \(\partial R\) must intersect \(\beta_m\), or else \(R\) would be isotopic to \(B'\), contradicting the definition of \(R\).
Further, since both endpoints of \(\beta_m\) must be on the same side of \(\partial R\), there must be an even number of points of intersection between \(\beta_m\) and \(\partial R\).
It follows that along the interior of \(\beta_*\), there must be at least four points of intersection with \(\partial R\).
Along \(\beta_*\), let \(c\) be the point of \(\beta_*^\circ\cap\partial R\) nearest to \(a\), and let \(d\) be the point of \(\beta_*^\circ\cap\partial R\) nearest to \(b\).

\begin{figure}
\centering
\includegraphics[width=.3\textwidth]{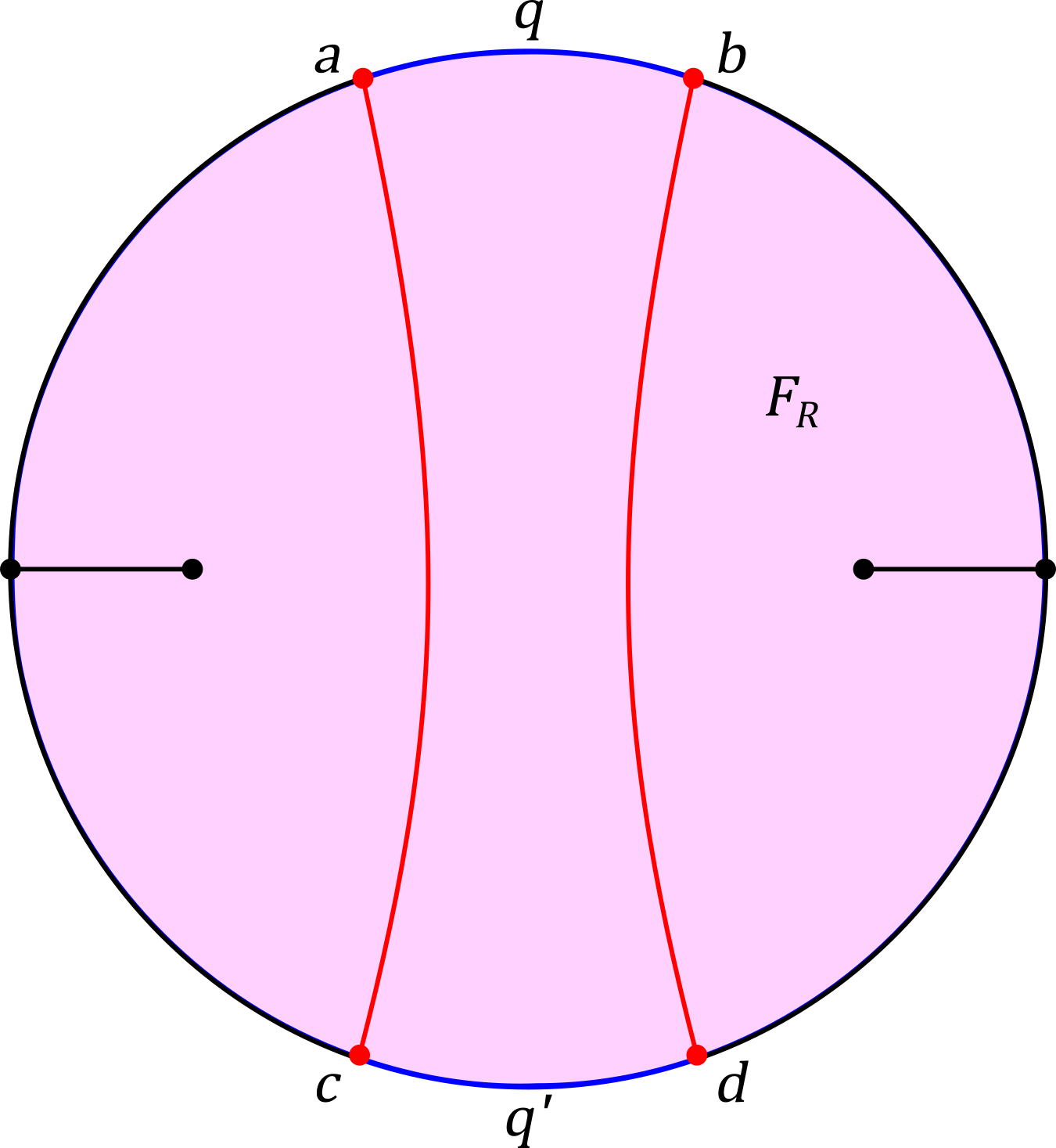}
\caption{The arc \(q\) in relation to the disk \(F_R\).
The outer circle is \(\partial R\), and the red vertical arcs are subarcs of \(\beta^*\).
}
\label{fig:220907_not-almost-carried-F_R}
\end{figure}

Since \(R\) is a cap, \(R\) cuts a 2-punctured disk \(F_R\) out of \(P_1\).
(See Figure \ref{fig:220907_not-almost-carried-F_R}.)
Consider the components of \(\left(\cup_{i=1}^m D_-^i\right)\cap F_R\).
There are two components which are arcs that connect a puncture in \(F_R\) to \(\partial R\) at points we will call \(x\) and \(y\).
All of the rest of arcs are parallel arcs which separate the two punctures of \(F_R\).

The points \(x\) and \(y\) cut \(\partial R\) into two arcs, one of which must contain both endpoints of \(q\); otherwise \(q\) would intersect \(\beta_m\) at point \(x\) or at point \(y\), contradicting the definition of \(q\).
Now in \(F_R\), \(a\) and \(b\) are connected via arcs of \(\partial R\) to the points \(c\) and \(d\), respectively.
These two arcs in \(F_R\) with endpoints \(a\), \(b\), \(c\), and \(d\), along with the arc \(q\) and another arc of \(\partial R\) form a quadrilateral in \(F_R\) whose interior is disjoint from \(\cup_{i=1}^m D_-^i\) (clear from Figure \ref{fig:220907_not-almost-carried-Q}).
Let the side of this quadrilateral whose endpoints are \(c\) and \(d\) be called \(q'\).

Now \(q'\) is parallel to \(q\).
But at this point we could repeat this argument, focusing on \(q'\) instead of on \(q\), which would lead us to accept the existence of another parallel arc \(q''\), and we could repeat this infinitely many times, each time obtaining another arc of \(\partial R\) with endpoints on \(\beta_{m-1}\), each of which is nested inside the last one.
But this contradicts general position; it cannot be the case that \(\beta_{m-1}\) cuts \(\partial R\) into infinitely many subarcs.
Therefore we conclude that \(R\) cannot be almost carried by \(\tau_1''\).

\begin{figure}
\includegraphics[width=.9\textwidth]{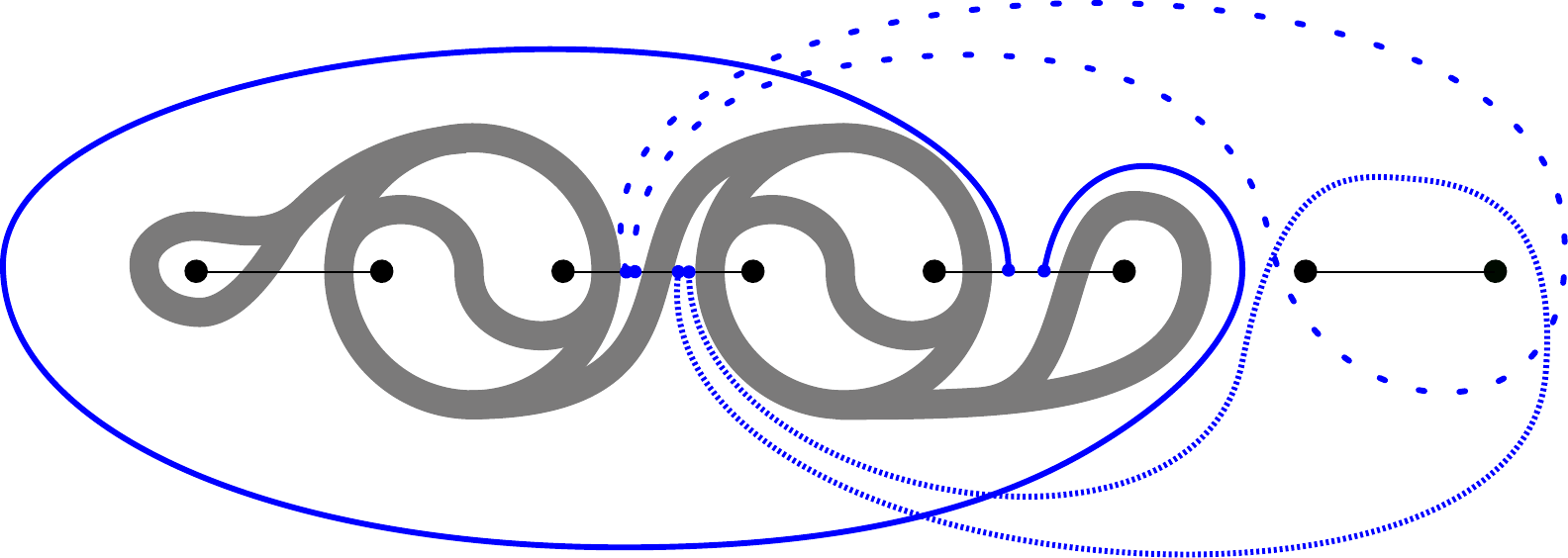}
\caption{When \(\partial R\) is almost carried by \(\tau_1'\), there are many possibilities for the arc \(q\).
Three are illustrated here (each with a different stroke style.)}
\label{fig:211207NotAlmostCarriedProof_case2.pdf}
\end{figure}

\begin{figure}
\includegraphics[width=.9\textwidth]{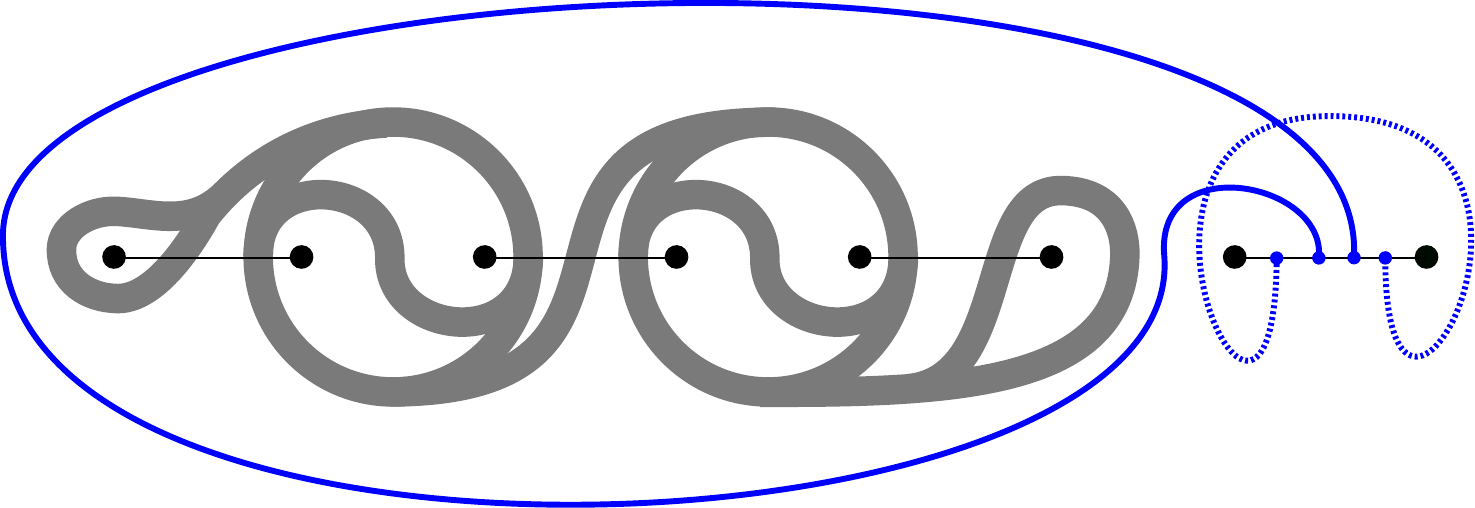}
\caption{Pictured here are the only non-isotopic possibilities for what \(q\) could look like, given that its endpoints lie on \(\beta_m\).
Either way, \(\partial R\backslash q\) must contain an arc that cobounds a bigon with \(\beta_m\), contradicting minimality. }
\label{fig:211207NotAlmostCarriedProof_case2-contra}
\end{figure}

Assume then that \(\partial R\) is almost carried by \(\tau_1'\).
In this case, there are more options for what the arc \(q\) may look like, but \(q\) still must be an arc with endpoints on a \(\beta\)-arc and with interior disjoint from any \(\beta\)-arcs.
(See Figure \ref{fig:211207NotAlmostCarriedProof_case2.pdf}.)
The arc \(q\) cannot have its endpoints on \(\beta_m\), as in Figure \ref{fig:211207NotAlmostCarriedProof_case2-contra}, because that would force \(\partial R\backslash q\) to contain an arc which bounds a bigon with \(\beta_m\); in other words, \(\partial R\) would not be in minimal position with respect to \(\beta_m\).
Then since \(\partial q\) lies in \(\beta_{i}\) for some \(i\in\{1,2,3,\dots,m-1\}\), then \(q\cup\beta_*\) bounds a disk \(Q\subseteq P_R\) which contains \(\beta_m\).
At this point we can apply the same logic as above, leading us to assert the existence of an infinite set \(\{q,q',q'',q''',\hdots\}\) of pairwise disjoint arcs of \(\partial R\) cut out by \(\beta_i\), which again contradicts general position.
Therefore, we may conclude that \(\partial R\) cannot be almost carried by \(\tau_1''\) either.
\end{proof}

\begin{lem}\label{prop:disjointblue}
The blue disks \(B\) and \(B'\) are a weak reducing pair for the bridge sphere.
\end{lem}
\begin{proof}
Observe that for all points along \(\partial B\), the \(y\)-coordinates are all less than \(4\).
(We will speak informally this way even though technically we mean that an isotopy class of \(\partial B\) in \(P_n\) has the property that all the \(y\)-coordinates are less than \(4\).)
Moving down a level,  $\pi_{n-1}(\partial B)$ is a loop in \(P_{n-1}\) whose \(y\)-coordinates are all less than \(5\).
Similarly, \(\pi_{n-2}(\partial B)\) is a loop in \(P_{n-2}\) whose \(y\)-coordinates are all less than \(6\), and so on.
In general, for the levels corresponding to \(D_a\), \(\pi_{n-k}(\partial B)\) is a loop in \(P_{n-k}\) whose \(y\)-coordinates are all less than \(4+k\).
Recall that \(n=n_a+n_b\), so \(n_b+1=n-n_a+1=n-(n_a-1)\).
Therefore \(\pi_{n_b+1}(\partial B)=\pi_{n-(n_a-1)}(\partial B)\) is a loop in \(P_{n-(n_a-1)}\) whose \(y\)-coordinates are all less than \(4+n_a-1\).
Since the dimensions of \(D_a\) were chosen so that \(n_a=2m-4\), it follows by substitution that the \(y\)-coordinates of \(\pi_{n_b+1}(\partial B)\) are all less than \(4+(2m-4)-1=2m-1\).
This means that \(\pi_{n_b+1}(\partial B)\) is completely to the left of the rightmost two punctures of \(P_{n_b+1}\).
Thus \(\pi_{n_b+1}(\partial B)\) is disjoint from \(\sigma_{n_b+1}(\beta_m)\) (the straight line segment connecting those two punctures).

Consider the \(\pi\)-projections of \(\partial B\) at consecutively lower levels.
For \(1\leq t\leq n_b\), $\pi_t(\partial B)$ will remain disjoint from $\sigma_t(\beta_m)$ because the isotopy $\pi_{t}$ from $P_{t+1}$ to $P_{t}$ fixes the two rightmost punctures of the bridge sphere pointwise.
Observe that \(\sigma_1(\beta_m)=\beta_m=D^m_-\cap P_1\).
Therefore $\pi_1(\partial B)$ is disjoint from \(\beta_m\), which implies that $\pi_1(\partial B) \cap \partial B' = \varnothing$, and so \(\{B,B\}\) are a weak reducing pair.
\end{proof}

\begin{obs}\label{obs:compress-a-cap}
In general, if we compress the cap \(R\) for a bridge arc \(\alpha\) along a boundary compressing disk \(\Delta\), the result will be two disjoint compressing disks whose boundary loops cut the bridge sphere into two punctured disks and a twice-punctured annulus, the latter of whose punctures correspond to \(\alpha\).
This means that neither of the resulting compressing disks are caps for \(\alpha\).
\end{obs}

\begin{lem}\label{lem:blue_intersects_reds}
The blue cap \(B\) for \(\alpha_+^1\) intersects all of the red disks below \(P_1\).
\end{lem}

\begin{proof}
First, by Proposition \ref{prop:corner-caps}, we already know that \(B\) intersects all of the red disks below \(P_1\) which are not caps for \(\alpha_-^m\), so it only remains to show that \(B\) also intersects all the red caps for \(\alpha_-^m\).

Observe that $\partial B'$ cuts $P_1$ into a disk $F_1$ with two punctures and another disk $F_2$ with $2m-2$ punctures.
By Corollary \ref{cor:what-B-and-reds-cover}, $\pi_1(\partial B)$ covers \(T_1'\) (the mini-graph defined in Proposition \ref{prop:what-below-what} consisting of all of \(T_1\) except the rightmost two eyelets).
It follows from Lemma \ref{prop:disjointblue} that \(\pi_1(\partial B)\) is contained in \(F_2\).

Let \(R\) be a red cap for \(\alpha_-^m\), and assume by way of contradiction that \(\{B,R\}\) is a weak reducing pair.
Arrange for \(R\) to be in minimal position with respect to the bridge disk \(D_-^m\).
An outermost disk \(\Delta\) on \(D_-^m\) cut out by an outermost arc of \(D_-^m\cap R\) must be a boundary compressing disk for \(R\), or else \(D_-^m\) and \(R\) would not be in minimal position.
We perform a boundary compression of \(R\) along \(\Delta\), resulting in a disjoint union \(R_1\sqcup R_2\) of two nonparallel compressing disks for \(P_1\).
Since \(R\) was disjoint from \(\partial B\), and the boundary compression happened away from \(\partial B\), both \(R_1\) and \(R_2\) are also disjoint from \(B\).
Further, by Observation \ref{obs:compress-a-cap}, neither \(R_1\) nor \(R_2\) are caps for \(\alpha_-^m\), and so we have two weak reducing pairs, \(\{B,R_1\}\) and \(\{B,R_2\}\) which both contradict Proposition \ref{prop:corner-caps}.

\end{proof}


\begin{lem}
If \(R_a\) and \(R_b\) are red disks above and below the bridge sphere (respectively), then \(\{R_a,R_b\}\) is not a weak reducing pair.
\end{lem}

\begin{proof}
Assume to the contrary that \(\{R_a,R_b\}\) is a weak reducing pair of red disks.
By Proposition \ref{prop:corner-caps}, \(R_a\) and \(R_b\) are caps for \(\alpha_+^1\) and \(\alpha_-^m\), respectively.
By Corollary \ref{cor:what-B-and-reds-cover}, the loop $\pi_1(\partial R_a)$ covers either \(T_1\) or some mini-graph \(T_1''\) constructed from \(T_1\) by excluding the leftmost and/or the rightmost eyelet of \(T_1\).

Suppose \(\partial R_b\) is disjoint from $\pi_1(\partial R_a)$. 
Then by Lemma \ref{lem:ACandDisjoint}, $\partial R_b$ is isotopic to an almost carried loop, which contradicts Lemma \ref{lem:not_almost_carried}. 
\end{proof}

The following lemma is an immediate corollary of Theorem 5.10 from \cite{pongtanapaisan2020critical} since the upper braid \(D_a\) of our link is a \((n_a,m)\) plat link with \(n_a=2m-4\). 

\begin{lem}
The cap \(B'\) is disjoint from all red disks on the other side of the bridge sphere.
\end{lem}




We have now shown that \(B\) and \(B'\) are the only weak reducing pair, and so we have proved our main theorem.

\vspace{1em}
\noindent
\textbf{Theorem 1.1} 
\textit{There exist infinitely many links with keen weakly reducible bridge spheres.
}

\vspace{1em}

Since a bridge sphere $\Sigma_L$ of a link $L$ induces a Heegaard surface $\widetilde{\Sigma_L}$ for the 2-fold cover of $S^3$ branched along $L$, it is natural to ask whether $\widetilde{\Sigma_L}$ satisfies properties that $\Sigma_L$ possesses. In our situation, the answer is no.

\begin{prop}
Keen weakly reducible bridge spheres in this paper do not lift to keen weakly reducible Heegaard surfaces.
\end{prop}
\begin{proof}
The blue compressing disk $B'$ is not only disjoint from $B$, but $\partial B'$ is also disjoint from a curve bounding a once-punctured disk $Q$ above $\Sigma_L$ (see Figure \ref{fig:cutdisk}). Such a once-punctured disk lifts to a compressing disk for one of the handlebodies of the Heegaard splitting of the double branched cover. Since $B'$ lifts to a compressing disk in the other handlebody whose boundary is disjoint from lifts of both $B$ and $Q,$ the Heegaard surface $\widetilde{\Sigma_L}$ is not keen.
\end{proof}

\begin{figure}[ht!]
\labellist
\small\hair 2pt

\pinlabel $\partial Q$ at 165 50
\endlabellist
      \centering
    \includegraphics[width=.5\textwidth]{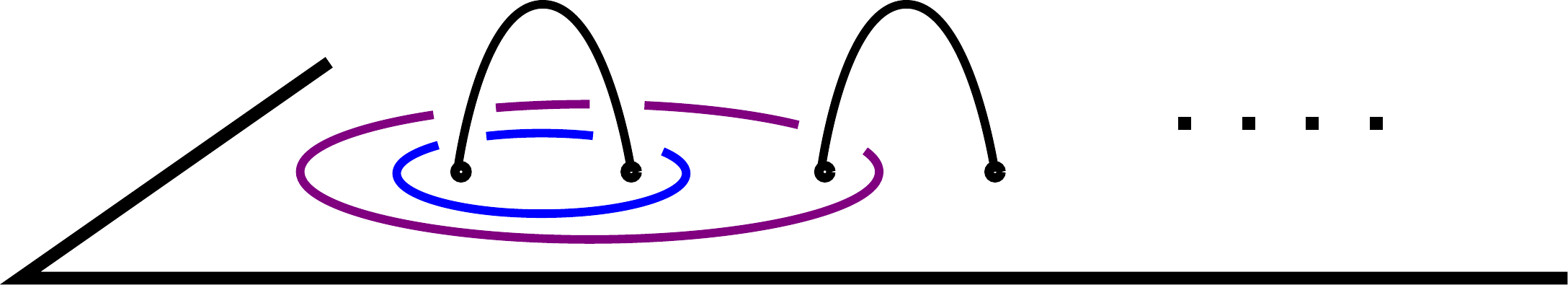}
    \caption{The purple curve bounds a once-punctured disk 
    \(Q\) above the bridge sphere and $\partial Q \cap \partial B' = \emptyset.$
    To visualize \(Q\), imagine a hemisphere shaped disk whose boundary is the purple loop and which is punctured by the second pictured bridge arc.}
    \label{fig:cutdisk}
\end{figure}
\section{Non-topologically minimal bridge spheres}
One of the main motivations of this article is to search for examples of bridge spheres that are not topologically minimal. The following criterion is needed for our construction of links with non-topologically minimal bridge spheres.
\subsection{Cho's Criterion}
For a link in bridge position, we have that $V_{\pm} \backslash N(\alpha_{\pm})$ is homeomorphic to a handlebody. Therefore, the complex spanned by compressing disks for the bridge sphere for $L$ in $V_{\pm} \backslash N(\alpha_{\pm})$ is a full subcomplex of the disk complex of the handlebody. We recall the following criterion by Cho \cite{cho2008homeomorphisms}:
\begin{thm}\label{thm:cho}
If $\mathcal{L}$ is a full subcomplex of the disk complex of the handlebody $\mathcal{K}(V_{\pm} \backslash N(\alpha_{\pm}))$ that satisfies the following condition, then $\mathcal{L}$ is contractible:
\begin{itemize}
    \item Let $D$ and $E$ be disks representing vertices of $\mathcal{L}$ and suppose that $D \cap E \neq \varnothing$. We assume that $D$ intersects $E$ minimally and transversely. If $\Delta \subset D$ is an outermost subdisk cut off by an outermost arc of $D \cap E$, then at least one of the disks obtained from surgery on $E$ along $\Delta$ is also a vertex of $\mathcal{L}.$
\end{itemize}
 \end{thm}
 \begin{prop}\label{prop:contractible}
The disk complex of $(V_{\pm},\alpha_{\pm})$ is contractible
\end{prop}
\begin{proof}
Suppose that compressing disks $D$ and $E$ in $(V_{\pm},\alpha_{\pm})$ intersect transversely and minimally. Then the boundary of one of the disks that arises from surgery on $E$ along $\Delta$ defined as in Theorem \ref{thm:cho} must enclose at least two punctures. Otherwise, $D \cap E$ would not be minimal.
\end{proof}
Using Cho's criterion and Theorem \ref{theorem:main}, we obtain the following corollary.

\begin{cor}
 There is an infinite family of nontrivial links with bridge spheres that are not topologically minimal.
\end{cor}
\begin{proof}
Since the bridge sphere $P_1$ for each link $L \in\mathcal{L}$ contains a unique pair of disjoint compressing disks on opposite sides of $P_1$, there is exactly one edge connecting the contractible disk complex of $(V_+,\alpha_+)$ to the contractible disk complex of $(V_-,\alpha_-)$ showing that the disk complex of $P_1$ is contractible.
\end{proof}

\bibliographystyle{plain}
\bibliography{ref}
\end{document}